\documentclass[12pt,reqno]{amsart}

\usepackage{graphicx}
\usepackage{mathrsfs}
\usepackage{color}
\usepackage{comment}
\usepackage{amssymb}
\usepackage{amsmath}
\usepackage{amscd}
\usepackage{physics}
\usepackage{mathtools}
\usepackage{tikz, tikz-cd}
\usepackage
[colorlinks=true, urlcolor=blue, citecolor=blue, linkcolor=blue]{hyperref}
\usepackage{enumitem}

\usepackage{microtype} 
\usepackage{lmodern}
\usepackage{csquotes}
\usepackage{nicefrac}

\calclayout

\let\oldtocsection=\tocsection
\let\oldtocsubsection=\tocsubsection
\let\oldtocsubsubsection=\tocsubsubsection

\renewcommand{\tocsection}[2]{\hspace{0em}\oldtocsection{#1}{#2}}
\renewcommand{\tocsubsection}[2]{\hspace{3em}\oldtocsubsection{#1}{#2}}
\renewcommand{\tocsubsubsection}[2]{\hspace{6em}\oldtocsubsubsection{#1}{#2}}

\numberwithin{equation}{section}

\theoremstyle{definition}
\newtheorem{theorem}{Theorem}[section]
\newtheorem{lemma}{Lemma}[section]
\newtheorem{prop}{Proposition}[section]
\newtheorem*{theorem*}{Theorem}
\newtheorem{corollary}{Corollary}[section]

\newtheorem{theorema}{Theorem}

\newtheorem{definition}{Definition}[section]

\newtheorem{quest}[]{Question}
\newtheorem{conjecture}[]{Conjecture}
\newtheorem{obs}{Observation}[section]
\newtheorem{exam}{Example}[section]
\newtheorem{remark}{Remark}[section]

\newcommand{\diam}{\mathrm{diam}}
\newcommand{\vol}{\mathrm{vol}}
\newcommand{\scal}{\mathrm{Sc}}
\newcommand{\msc}{\mathrm{mscal}}
\newcommand{\sq}{\mathrm{Sq}^}

\newcommand{\uw}{\mathrm{UW}_ }

\newcommand{\RR}{\mathbf{R}}
\newcommand{\ZZ}{\mathbf{Z}}

\begin{document}
\title{Urysohn width and macroscopic scalar curvature}
\author{Aditya Kumar}
\address{Department of Mathematics, University of Maryland, 4176 Campus Dr, College
Park, MD 20742, USA}
\email{akumar65@umd.edu}

\author{Balarka Sen}
\address{School of Mathematics, Tata Institute of Fundamental Research. 1, Homi Bhabha Road, Mumbai-400005, India}
\email{balarka2000@gmail.com, balarka@math.tifr.res.in}
\begin{abstract}
    We show that the macroscopic version of Gromov’s Urysohn width conjecture for scalar curvature is false in dimensions four and above. This is based on (1) a novel estimate on the codimension two Urysohn width of circle bundles over manifolds with large hypersphericity radius, and (2) a notion of ruling for Riemannian manifolds that yields circle bundles with total spaces admitting metrics of positive macroscopic scalar curvature. Along the way, we also show that Urysohn width is not continuous under Cheeger-Gromov collapsing limits. This article is a continuation of our study of metric invariants and scalar curvature for circle bundles over large Riemannian manifolds initiated in \cite{pscbundle}.
\end{abstract} 

\maketitle

\tableofcontents

\section{Introduction}
\subsection{Background and motivation}
\begin{definition}[Urysohn width]
    Let $X$ be a metric space. The Urysohn $k$-width of $X$, denoted by $\mathrm{UW}_k(X)$, is the infimum over all $t > 0$ such that there is a continuous map $f: X \to K$ to a simplicial complex $K$ of dimension at most $k$ satisfying the property that the diameter of each fiber of $f$ is at most $t$, i.e., for all $p \in K$, $\mathrm{diam}(f^{-1}(p)) \leq t$.
\end{definition}

The notion of Urysohn width originated in the works of Pavel Urysohn on topological dimension theory \cite{ury26,ale33}. Heuristically, $\uw k(X) \leq \varepsilon$ indicates that the metric space $X$ looks like a $k$-dimensional simplicial complex upon ignoring features at scale $\varepsilon$ or less. In Riemannian geometry,  the notion of Urysohn width was popularized by several works of Gromov \cite{gromovwidth,glarge}. It arises naturally in the study of non-collapsing and thick-thin decompositions and has also played an important role in systolic geometry. Some of the recent works on these topics have been \cite{guth2017,nabu,papu,sabu,bb24,abg24}. For a comprehensive treatment of Urysohn width, we refer the reader to the thesis of Balitskiy \cite{balthesis}.

\begin{definition}[Scalar curvature]
    Let $(M^d, g)$ be a Riemannian manifold. The scalar curvature of $M$ at a point $x$, denoted by $\mathrm{Sc}_M(x)$, is the coefficient appearing in the $(d+2)$\textsuperscript{th} order Taylor expansion of the volume of an infinitesimal ball centered at $x$:
    $$\vol_M(B_r(x)) = \omega_d r^d \left (   1 - \frac{\scal_M(x)}{6(d+2)}r^2 + O(r^3)\right ),$$
    where $\omega_d$ denotes the volume of a unit ball in the Euclidean $d$-space.
\end{definition}

The scalar curvature can be computed infinitesimally as the trace of the Ricci curvature tensor $\mathrm{Sc}(x) = \mathrm{tr}~ \mathrm{Ric}_x$. The study of the geometry and topology of manifolds possessing \emph{positive} scalar curvature metrics has seen an explosion of activity in recent years. The primary motivation for this article is the following outstanding conjecture of Gromov on the Urysohn width of Riemannian manifolds with positive scalar curvature (see \cite[2.A]{glarge}, \cite[G4]{gromovwidth}, \cite[2\textonehalf]{Gro96}).

\begin{conjecture}[Gromov]\label{conj-gro}
    Let $M^d$ be a complete Riemannian manifold of dimension $d$. If $M$ admits a metric with scalar curvature $\mathrm{Sc} \geq \sigma^2 > 0$, then there exists a dimensional constant $C_d$ such that $\mathrm{UW}_{d-2}(M) \leq C_d/\sigma$.
\end{conjecture}

The primary evidence in this direction is the following simple observation: 

\begin{exam}\label{eg-xs2}
    Given any closed Riemannian manifold $M^{d-2}$, consider its product with a round $2$-sphere of radius $\varepsilon$, i.e., the manifold $X^d = M^{d-2} \times S^2({\varepsilon})$ equipped with the product metric. Then, for $\varepsilon > 0$ sufficiently small, the scalar curvature of $X$ satisfies:
    $$\mathrm{Sc}_X = \mathrm{Sc}_M + 2 \varepsilon^{-2} \geq \varepsilon^{-2}.$$
    Fibers of the projection $X \to M$ have diameter $\mathrm{diam}(S^2(\varepsilon)) = \pi \varepsilon$. Thus, $\mathrm{UW}_{d-2}(X) \lesssim \varepsilon.$
\end{exam}

Conjecture \ref{conj-gro} asserts that every complete manifold with a uniform positive lower bound on the scalar curvature exhibits the same behavior as Example \ref{eg-xs2}. One could view Conjecture \ref{conj-gro} as the assertion that if the volume of infinitesimal balls in $(M, g)$ is less than that of Euclidean balls of the same radius, then the codimension $2$ Urysohn width of $M$ is uniformly bounded. So far, Conjecture \ref{conj-gro} is only known to hold in dimension 3 \cite{lm,cl}. In fact, even the codimension $1$ width case involving a uniform upper bound on $\uw{d-1}(M^d)$ remains unknown.

If instead of an upper bound on the volume of balls with \emph{infinitesimal radius}, one has an upper bound on the volume of balls of \emph{unit radius}, then this provides a large scale or \emph{macroscopic} notion of positivity of scalar curvature. The following remarkable result due to Guth \cite{guth2006,guth2017} (see also, \cite{papu,nabu}) deduces a uniform upper bound on $\uw{d-1}(M)$ if $M$ has sufficiently positive macroscopic scalar curvature, i.e., the volume of every unit ball in $M$ is bounded above by a sufficiently small dimensional constant. 
\begin{theorem}[Guth]\label{thm-guth-codim1} 
    There exists a constant $\delta_d >0$ such that for any closed Riemannian manifold $(M^d,g)$, if every ball of radius $1$ in $M$ has volume at most $\delta_d$, then $\uw{d-1}(M) \leq 1$.
\end{theorem}

The notion of macroscopic scalar curvature and its analogy with scalar curvature was discussed at length by Guth in his ICM address \cite[Section 7]{guthicm}. In particular, he states the metaphor that ``the macroscopic scalar curvature is like the scalar curvature." \cite[Metaphor 3, pg.~11]{guthicm}.
In consonance with this metaphor, in the past few years, there have been several works resolving macroscopic analogs of known as well as conjectural results for scalar curvature \cite{guthsys,af17,bk19,bs21,alp22,tgmm}. Nonetheless, as observed in \cite{abg24}, a straightforward generalization of Theorem \ref{thm-guth-codim1} to codimension $2$ Urysohn width does not hold. We give two examples to illustrate this.

\begin{exam}\label{eg-badegs}
    Let $S^d(r)$ denote a $d$-sphere with the round metric of radius $r$. Let $\varepsilon > 0$ be very small and $R \gg 1$ be very large.
    \begin{enumerate}
        \item\label{eg-badeg1} Consider the manifold $S^1(\varepsilon) \times S^2(R)$, with the product metric. Unit balls of this manifold are roughly isometric to $S^1(\varepsilon) \times D^2(1)$. Hence, they have volume $\sim O(\varepsilon)$. However, by the fiber contraction lemma (Lemma \ref{lem-fcl}), $\mathrm{UW}_1(S^1(\varepsilon) \times S^2(R)) \sim R$.
        \item\label{eg-badeg2} Consider $S^3$ as the total space of the Hopf fibration $S^1({\varepsilon}) \to S^3 \xrightarrow{\pi} S^2(R)$ equipped with the Berger metric $g_{\varepsilon, R} = \pi^* g_{S^2(R)} + \varepsilon^2 \theta \otimes \theta$. Once again, unit balls of $(S^3, g_{\varepsilon, R})$ are roughly isometric to $S^1(\varepsilon) \times D^2(1)$. Therefore, they have volume $\sim O(\varepsilon)$. Since $\pi$ is not nullhomotopic, the fiber contraction lemma implies $\mathrm{UW}_1(S^3, g_{\varepsilon, R}) \sim R$.
    \end{enumerate}
\end{exam}      

Note that as scalar curvature is an infinitesimal quantity, a bound of the form $\mathrm{Sc} \geq \sigma^2$ on $(M, g)$ continues to hold upon passage to the universal cover $(\widetilde{M}, \widetilde{g})$. The first example is eliminated by imposing a control on the volumes of balls in the universal cover. Indeed, the universal cover of $S^1(\varepsilon) \times S^2(R)$ is $\mathbf{R} \times S^2(R)$, and unit balls here have substantial volume. However, it does not eliminate the second example, as $S^3$ is simply connected. 

To remedy this, Alpert-Balitskiy-Guth \cite{abg24} propose that one should impose a control on the volume of unit balls in the universal cover of a slightly larger concentric ball. {We reproduce their proposal as a definition below. Note that the following definition slightly differs from the notion of macroscopic scalar curvature defined in \cite[Section 7]{guthicm}.}

\begin{definition}[Macroscopic scalar curvature]\label{def-msc}
    Let $(M^d,g)$ be a Riemannian manifold. Let $V^d_r(\sigma)$ denote the volume of a ball of radius $r$ in a simply connected Riemannian space form of dimension $d$ and constant scalar curvature $\sigma$. 
    \begin{itemize}
        \item Let $x \in M$ and let $B_{2r}(x)$ be the $2r$-ball centered at $x$.
        \item Let $\widetilde{B}_{2r}(x)$ be the universal cover of $B_{2r}(x)$ equipped with the pulled back metric.
        \item Select $\tilde x \in \widetilde{B}_{2r}(x)$ lying over $x$. Let $B_r(\tilde x) \subset \widetilde{B}_{2r}(x)$ denote the $r$-ball centered at $\tilde x$.
    \end{itemize}
    The macroscopic scalar curvature at $x$ and at scale $r > 0$, denoted by $\msc(x; r)$, is defined to be the unique $\sigma \in \RR$ such that $\vol(B_r(\tilde x)) = V^d_r(\sigma)$. Fixing the scale at $r = 1$, we shall simply say the macroscopic scalar curvature at $x$ is $\msc(x) := \msc(x; 1)$.
\end{definition}

\begin{remark}\label{rem-mscprop} We summarize some basic properties of macroscopic scalar curvature. 
\begin{enumerate}
    \item For any fixed $r > 0$, $V^d_r(\sigma)$ is a decreasing function of $\sigma$. Moreover, $\lim_{\sigma \to \infty} V^d_r(\sigma) = 0$ and $\lim_{\sigma \to -\infty} V^d_r(\sigma) = \infty$. Therefore, 
    \begin{equation*}
        \msc(x; r) \geq \sigma \iff \vol(B_r(\widetilde{x})) \leq V^d_r(\sigma).
    \end{equation*}
    \item Macroscopic scalar curvature scales like the scalar curvature. More precisely, let $(M, g)$ be a Riemannian manifold and $\lambda > 0$. Then, for all $r > 0$, 
    $$\msc_{\lambda^2 g}(x; \lambda r) = \lambda^{-2} \cdot \msc_{g}(x; r).$$
    \item At infinitesimal scales, macroscopic scalar curvature is equal to scalar curvature, i.e.,
    $$\lim_{r \to 0} \msc(x; r) = \mathrm{Sc}(x).$$
\end{enumerate}
\end{remark}

In this article, we will address the macroscopic version of Conjecture \ref{conj-gro}, as formulated by \cite[Conjecture 9]{abg24}. Our formulation is equivalent to theirs via an appropriate scaling and using Property (3) of Remark \ref{rem-mscprop}.

\begin{conjecture}\label{conj-abgtrue}
 There exists dimensional constants $C_d, \sigma_d > 0$ such that the following holds. Let $M^d$ be any closed Riemannian manifold with $\msc \geq \sigma_d$. Then, $\mathrm{UW}_{d-2}(M) \leq C_n$.
\end{conjecture}

In contrast with Conjecture \ref{conj-gro}, this conjecture has not yet been established even in dimension $3$. However, partial results in dimension $3$ are known for a version of Conjecture \ref{conj-abgtrue} in \cite[Conjecture 3]{abg24}, wherein the authors impose a control on the volume of balls on the manifold as well as an additional local acyclicity control condition, i.e., that every loop in a unit ball bounds a $2$-chain supported in a concentric ball of slightly larger radius. See, \cite[Theorem 6]{abg24}. 

\subsection{Main results} The primary purpose of this article is to construct counterexamples to Conjecture \ref{conj-abgtrue}. Our main theorem is as follows.

\begin{theorema}[Theorem \ref{thm-final}]\label{thm-mainres}
    Let $d \geq 4$. For any $\sigma > 0$ and positive integer $n$, there are $d$-dimensional closed Riemannian manifolds $E_n$ satisfying $\msc \geq \sigma$ and $\mathrm{UW}_{d-2}(E_n) \gtrsim n$. Therefore, Conjecture \ref{conj-abgtrue} is false in dimensions $4$ and above.
\end{theorema}

The manifolds $E_n$ we construct are total spaces of principal circle bundles
$$S^1(\varepsilon) \to E_n \stackrel{\pi}{\to} M_n,$$
where $M_n^{d-1}$ are certain Riemannian manifolds with $\mathrm{UW}_{d-2}(M_n) \gtrsim n$. The metric on $E_n$ is
$$g_{E_n} = \pi^* g_{M_n} + \varepsilon^2 \theta \otimes \theta,$$
for some auxiliary choice of connection $\theta$ on $\pi$. We call this metric an \emph{adapted metric} (cf. Definition \ref{def-adapted}). By choosing $\varepsilon >0$ arbitrarily small, we may ensure that volumes of unit balls in $E_n$ are arbitrarily small as well, as in Example \ref{eg-badegs}. However, we observe that:
\begin{enumerate}
    \item\label{item-uw} $\mathrm{UW}_{d-2}(M_n)$ being large does not imply $\mathrm{UW}_{d-2}(E_n)$ is large, even if $\varepsilon$ is small.
    \item\label{item-msc} Volumes of unit balls in $E_n$ being small does not imply $\msc_{E_n}$ is large.
\end{enumerate}
The second observation is precisely the content of Example \ref{eg-badegs} and the discussion following it, which illustrates that the definition of macroscopic scalar curvature is tailored to exclude these examples. 

\begin{remark}\label{rem-cgcollapse}
    The first observation might come across as surprising since, as $\varepsilon \to 0$, $E_n$ converges to $M_n$ in the Cheeger-Gromov sense. Therefore, one might expect $\mathrm{UW}_{d-2}(E_n) \to \mathrm{UW}_{d-2}(M_n)$, as $\varepsilon \to 0$. However, this is not always true! In \cite{pscbundle}, we had constructed examples of circle bundles over enlargeable manifolds such that the universal cover of their total spaces admit an equivariant map to a simplicial complex of codimension at least two, of uniformly bounded diameter fibers. Passing to large finite sheeted covers, it is possible to ensure the total space has uniformly bounded width and the circle fibers have arbitrarily small radius, yet the base manifold has arbitrarily large width. See, Proposition \ref{prop-pscbundles}.
\end{remark}

One of the key results in this article, used to estimate $\mathrm{UW}_{d-2}(E_n)$, is a lower bound on the codimension 2 Urysohn width of total spaces of principal circle bundles equipped with adapted metrics, in terms of the $\ZZ_2$-hypersphericity (denoted by $\mathrm{HS}(\cdot;\ZZ_2)$; cf. Definition \ref{def-hyp}) of the base manifold, provided the base manifold admits a $\mathrm{Pin}^-$ structure. In precise terms, we prove the following. 

\begin{theorema}[Theorem \ref{cor-fcl3}]\label{thm-uwestimate}
    Let $M^{d-1}$ be a closed Riemannian manifold which admits a $\mathrm{Pin}^-$ structure. Let $\pi : E^d \to M^{d-1}$ be a principal circle bundle over $M$. We equip $E$ with any adapted metric (cf. Definition \ref{def-adapted}). Then, $$\uw{d-2}(E) \geq \frac12 \mathrm{HS}(M; \mathbf{Z}_2).$$
\end{theorema}

$\mathrm{Pin}^-$ structures are a non-orientable variant of spin structures: an orientable $\mathrm{Pin}^-$ manifold is spin. This result is optimal in the sense that the conclusion is false if either the $\mathrm{Pin}^-$ assumption is dropped, or $\ZZ_2$-hypersphericity is replaced by $\ZZ$-hypersphericity. See Remark \ref{rem-pscbundles} for an explicit example. Theorem \ref{thm-uwestimate} is a sophisticated version of the width estimate in Example \ref{eg-badegs}, (\ref{eg-badeg2}). We prove it by combining the fiber contraction lemma (Lemma \ref{lem-fcl}) with obstructions coming from the generalized mod 2 Hopf invariant (Definition \ref{def-hopf}).

\begin{remark}
We find the appearance of the $\mathrm{Pin}^-$ condition above rather intriguing. The relationship between scalar curvature and spin structures, mediated by the Dirac operator, is quite well known. But it is generally believed to be an appendage of using Dirac operator techniques. However, in Theorem \ref{thm-uwestimate}, the $\mathrm{Pin}^-$ condition appears as a necessary condition in a result concerning Urysohn width, without any input whatsoever from spin geometry.
\end{remark}

It is known that any Kahler manifold $(X, \omega, g)$ with $\mathrm{Sc}_g > 0$ is uniruled, i.e., there is a rational curve passing through any point of $X$ \cite{hwruling}. The converse is also conjectured to be true. Motivated by the discussion in \cite[pg.~10]{grom}, we introduce a soft notion of ruling for Riemannian manifolds (cf. Definition \ref{def-ruling}): 

\begin{definition}
    A \emph{$\lambda$-small ruling} $\mathfrak{S} = \{\Sigma_x\}$ of a Riemannian manifold $M$ is a collection of smoothly embedded $2$-spheres $\Sigma_x \cong S^2$ passing through each point of $M$, such that $\diam(\Sigma_x) \leq \lambda$, and there exists $\omega \in H^2(M; \mathbf{Z})$ such that $\langle \omega, [\Sigma_x]\rangle = \pm 1$. We say $\omega$ \emph{tames} $\mathfrak{S}$.
\end{definition}

Certainly, a uniruling is an example of a ruling in this sense. Since $\omega \in H^2(M;\ZZ)$, there is a principal circle bundle $\pi : E \to M$ naturally associated with $\omega$, satisfying $c_1(E) = \omega$. Furthermore, since $\langle \omega, [\Sigma_x]\rangle = \pm 1$, the restriction of $E$ over any ruling sphere is the Hopf fibration. We call $E$ the circle bundle \emph{associated} to the ruling. Our second key result in this article, used to ensure $E_n$ has uniformly positive macroscopic scalar curvature, is the following. 

\begin{theorema}[Theorem \ref{thm-circbdlmscal}]\label{thm-mscabundle}
    Let $M$ be a closed Riemannian manifold that admits a $3/2$-small ruling. For any $\sigma > 0$, there exists a small enough fiber radius $\varepsilon > 0$ such that the total space $E$ of the circle bundle $S^1({\varepsilon}) \to E \to M$ associated with the ruling, equipped with the adapted metric (cf. Definition \ref{def-adapted}) satisfies $\msc \geq \sigma$.
\end{theorema}

{From Theorem \ref{thm-uwestimate} and Theorem \ref{thm-mscabundle}, we see that Theorem \ref{thm-mainres} follows, provided the existence of Riemannian manifolds $M_n^{d-1}$ satisfying the following conditions:}
\begin{enumerate}
    \item $M_n$ should admit a $\mathrm{Pin}^-$ structure, and satisfy $\mathrm{HS}(M_n;\ZZ_2) \gtrsim n$.
    \item $M_n$ should admit a $3/2$-small ruling.
\end{enumerate}
We shall provide two separate constructions of such a family of manifolds. 

\begin{enumerate}[topsep=1ex,itemsep=1ex,partopsep=0.5ex,parsep=0.5ex,label=\textbf{Construction \arabic*.}, wide=0pt]
\item \label{eg-1} Let $N^{d-1}$ be any $\mathbf{Z}_2$-enlargeable {(cf. Definition \ref{def-z2en})}, $\mathrm{Pin}^-$ manifold of dimension $d-1 \geq 3$. Then, there exists a cover $N_n \to N$ such that $\mathrm{HS}(N_n; \mathbf{Z}_2) \geq n$. Choose a uniform small constant $\delta > 0$. Let $M_n$ be obtained from $N_n$ by taking a connect sum with $S^2 \times S^{d-3}$ at each point of a maximal $\delta$-separated collection of points on $N_n$. Then $M_n$ admits a (highly singular) small ruling by smoothly embedded $2$-spheres of diameters equal to, upto an error of the order of $\delta$, the diameter of the round $2$-spheres $S^2 \times \{*\}$ that rule $S^2 \times S^{d-3}$, which we may choose to be less than $3/2$. The class $\omega \in H^2(M_n; \mathbf{Z})$ may be chosen to be the Poincare dual of the linear combination of the homology classes $[\{*\} \times S^{d-3}]$, one from each $S^2 \times S^{d-3}$ connect summand. For details, see Section \ref{sec-corrugration}.

\item \label{eg-2} Consider the $4$-manifolds $M_n$ constructed by Balitskiy \cite[Theorem 2.4.4]{balthesis}, elaborated further in \cite[Example 11]{abg24}. These are non-orientable quotients of a $4$-dimensional analogue of the Schwarz minimal surface, given by the locus of points in the Euclidean $5$-space equidistant from the standard $2$-skeleton and its dual. The small ruling on $M_n$ is obtained from realizing it as the boundary of a neighborhood of either one of the skeleta, ruled by the $2$-spheres given by the normal links to the $2$-dimensional cells of the skeleta. It follows from \cite[Theorem 2.4.4]{balthesis} and \cite[Example 11]{abg24} that these manifolds have large hypersphericity. Further examples in higher dimensions are obtained by taking product of $M_n$ with round circles of large radius. For details, see Section \ref{sec-abg}.
\end{enumerate}

These constructions provide counterexamples to Conjecture \ref{conj-abgtrue}, which proves Theorem \ref{thm-mainres}.

\subsection{Comparison with previous work}\label{sec-comparison} We begin by recalling a question of Gromov \cite[pg.~661]{Gro18} on the existence of circle bundles over large manifolds whose total spaces carry positive scalar curvature metrics. We reproduce a paraphrased version below. 

\begin{quest}[Gromov]\label{qn-gro}
Let $M^{d-1}$ be an enlargeable manifold. Let $S^1 \hookrightarrow E^d \xrightarrow{\pi} M^{d-1}$ be a principal circle bundle. Can $E$ admit a positive scalar curvature metric?
\end{quest}

In our earlier work \cite{pscbundle}, we answered Question \ref{qn-gro} in the affirmative, in dimension four and above. Further, we could ensure that the examples of the circle bundles constructed in \cite{pscbundle} have base manifolds that are spin. We now state a refinement of Gromov's question using the notion of enlargeability with $\mathbf{Z}_2$-coefficients. 

\begin{quest}\label{question1}
    Let $M^{d-1}$ be a $\mathbf{Z}_2$-enlargeable manifold (cf. Definition \ref{def-z2en}) which admits a $\mathrm{Pin}^-$ structure. Let $S^1 \hookrightarrow E^d \xrightarrow{\pi} M^{d-1}$ be a principal circle bundle. Can $E$ admit a positive scalar curvature metric?
\end{quest}

The main result of this article gives an affirmative answer to the macroscopic scalar curvature analog of Question \ref{question1}, in dimensions four and above. Our interest in Question \ref{question1} stems from the following observation, which relates Gromov's Question \ref{qn-gro} with Gromov's Conjecture \ref{conj-gro}.  

\begin{obs}
    Conjecture \ref{conj-gro} implies that the answer to Question \ref{question1} is negative.
\end{obs}
\begin{proof}
Suppose the answer to Question \ref{question1} is positive. Thus, there exists a metric $g^+$ on $E$ with $\mathrm{Sc}(g^+) > 0$. Let $g$ be an adapted metric (cf. Definition \ref{def-adapted}) on $E$. Since $E$ is compact, there exists $\sigma^2, C > 0$ such that $\mathrm{Sc}(g^+) \geq \sigma^2$, and $g$ and $g^+$ are $C$-bilipschitz equivalent. 

Since $M$ is $\mathbf{Z}_2$-enlargeable, it admits finite sheeted covers $M_n \to M$ with $\mathrm{HS}(M_n; \mathbf{Z}_2) \geq n$. Pulling back $\pi$ to $M_n$, we obtain a circle bundle $\pi_n : E_n \to M_n$. Observe that $E_n \to E$ is also a finite sheeted cover. Let $g_{n}^+, g_{n}$ denote the lifts of the metrics $g^+, g$ to $E_n$. Furthermore, $g_n$ is also an adapted metric on $E_n$. Since the covering map is a local isometry, $\mathrm{Sc}(g^+_{n}) \geq \sigma^2$, and $g_{n}^+$ and $g_{n}$ are $C$-bilipschitz equivalent, with the same constants $C, \sigma^2 > 0$. Therefore, $\uw{d-2}(E_n,g^+_n)$ and $\uw{d-2}(E_n,g_n)$ are uniformly comparable for all $n$. However, by Theorem \ref{thm-uwestimate}, $\mathrm{UW}_{d-2}(E_n, g^+_{n}) \gtrsim n$. As $\mathrm{Sc}(g^+_{n}) \geq \sigma^2$, this contradicts Conjecture \ref{conj-gro}. 
\end{proof}

We summarize, to the best of our knowledge, the existing results providing a negative answer to Question \ref{question1} under various conditions on $M$. In these cases, $M$ is orientable; as such, $\mathbf{Z}_2$-enlargeability of $M$ is equivalent to the following: for any $n \gg 1$, there exists a finite cover $M_n \to M$ such that $M_n$ admits a $1$-Lipschitz map of \emph{odd} degree to a sphere of radius $n$.

\begin{enumerate}
    \item A special class of $\mathbf{Z}_2$-enlargeable orientable manifolds are $1$-enlargeable manifolds\footnote{$1$-enlargeable means that the maps to the sphere in the definition of enlargeability are of degree $1$.}. He \cite[Proposition 4.13]{He25} proves that circle bundles with non-nullhomologous fibers over $1$-enlargeable manifolds of dimension $d - 1$ do not admit positive scalar curvature metrics for $d \leq 7$.        
    \item Suppose the fibers of $\pi : E \to M$ are null-homotopic in $E$, and $G := \pi_1E \cong \pi_1 M$ is a torsion-free group of cohomological dimension at most $d-1$, that also satisfies the Strong Novikov Conjecture. Let $E_n$ denote the pullback of $E$ to $M_n$. If $E$ admits a PSC metric, then combining the main result of Bolotov \cite{bol09} with the argument of \cite[Lemma 3.1]{pscbundle}, we get $\uw{d-2}(E_n) < C$ for some uniform $C > 0$. However, Theorem \ref{thm-uwestimate} implies $\uw{d-2}(E_n) \gtrsim n$, which gives a contradiction.
\end{enumerate}  

The examples outlined in Constructions \hyperref[eg-1]{1} and \hyperref[eg-2]{2} provide a sizable class of manifolds satisfying the hypotheses of Question \ref{question1}. In a few cases, Bolotov's result in (2) can be used to rule out the examples from Construction \hyperref[eg-1]{1}. However, if for the initial $\mathbf{Z}_2$-enlargeable manifold $N$ in Construction \hyperref[eg-1]{1}, $\pi_1 N$ has large cohomological dimension or torsion, we do not know the answer to Question \ref{question1}. Furthermore, Section \ref{sec-abg} provides an example of a circle bundle $\pi : E \to M$ over a non-orientable, $\mathbf{Z}_2$-enlargeable, $\mathrm{Pin}^-$ manifold $M$. As shown by Li-Zhang \cite{lz}, $M$ does not admit a positive scalar curvature metric, yet its orientation double cover does. But the answer to Question \ref{question1} for $E$ remains open. 
 
\par\medskip

\noindent {\bf Symplectic analogies.} In \cite{pscbundle}, the examples that answer Question \ref{qn-gro} affirmatively were constructed using symplectic geometry. More precisely, as our base manifolds, we used Donaldson divisors $Z_k \subset (M, \omega)$, Poincar\'{e} dual to $k[\omega]$ for $k \gg 1$, where $(M, \omega)$ is a symplectic manifolds satisfying various hypotheses needed to ensure $Z_k$ is enlargeable. See Proposition \ref{prop-pscbundles} for a summary. For ease of exposition, we call $k$ the \emph{twisting parameter}.

A vivid picture of Donaldson divisors, with large twisting parameter $k$, was shared with us by Dennis Sullivan \cite{dennis}: they look like a very fine mesh inside the symplectic manifold, smoothed near the intersections. As a concrete example, let $\Sigma_g$ denote a surface of genus $g$ equipped with an area form of unit area. We equip $\Sigma_g \times \Sigma_h$ with the product symplectic form. Then, Donaldson divisors $Z_k \subset \Sigma_g \times \Sigma_h$ may be constructed by plumbing several disjoint copies of $\{*\} \times \Sigma_h$ with several disjoint copies of $\Sigma_g \times \{*\}$, at the points of intersection. The number of these slices increase with $k \to \infty$. In general, Donaldson \cite[Proposition 40]{don96} showed that $k^{-1} Z_k$ converges to $\omega$ in the sense of currents. This illustrates that the mesh in Sullivan's picture becomes increasingly finer as the twisting parameter increases.

It is intriguing to us that a similar picture may be used to describe Construction \hyperref[eg-2]{2} where the equidistant hypersurface $M_n$ corresponds to the mesh, and parameter $n$ corresponds to the twisting parameter $k$. Similar considerations apply for the manifolds $M_n$ in Construction \hyperref[eg-1]{1}. For instance, in both the case of Donaldson divisors $Z_k$ and the manifolds $M_n$ in Constructions \hyperref[eg-1]{1} and \hyperref[eg-2]{2}, increasing parameters $k$ and $n$ respectively increases the Euler characteristic (for a short computation, see \cite[Proposition 4]{pscbundle}). However, the process of corrugating a manifold by attaching many trivial codimension $3$ handles at very fine scales as in Construction \hyperref[eg-1]{1} is perhaps more closely reminiscent of softer constructions in symplectic topology involving Gromov's $h$-principle.

\subsection{Organization}
In Section \ref{sec-fclhopf}, we prove a series of lemmas leading up to the proof of Theorem \ref{thm-uwestimate}. We end the section with with Proposition \ref{prop-pscbundles}, showing that Urysohn width is not continuous with respect to Cheeger-Gromov collapsing limits. In Section \ref{sec-strat}, we introduce the notion of ruling for Riemannian manifolds and prove Theorem \ref{thm-mscabundle}. In Section \ref{sec-construction} we provide details of Constructions \hyperref[eg-1]{1} and \hyperref[eg-2]{2} and establish some key properties of the manifolds constructed therein. In Section \ref{sec-main} we combine all these results to establish Theorem \ref{thm-mainres}. Finally, Appendix \ref{sec-appendix} contains some computational details pertaining to Section \ref{sec-abg}.

\subsection*{Acknowledgments}
The authors are grateful to Alexey Balitskiy and Larry Guth for suggesting several ideas and patiently answering their questions at various stages of this work. They are also indebted to Bernhard Hanke and Shihang He for helpful correspondence that led to several key ideas in this article. In particular, He pointed out the article \cite{bol06} to us, which served as the inspiration for Section \ref{sec-fclhopf}. They are also grateful to Mike Miller Eismeier for help with the algebraic topology in Section \ref{sec-fclhopf}, {as well as for suggesting several improvements on an earlier draft}. They are also thankful to Dennis Sullivan for sharing his insights on symplectic geometry, some of which are behind Section \ref{sec-comparison}.

The first author is grateful to Jonathan Rosenberg for several insightful conversations, particularly on Question \ref{question1}. He would also like to thank Or Hershkovits and Boyu Zhang for their interest in this work. The second author is grateful to his advisor Mahan Mj for his interest and several helpful comments on an earlier draft of this article, as well as his friends and colleagues Ritwik Chakraborty and Sekh Kiran Ajij for several helpful conversations. The second author is supported by the Department of Atomic Energy, Government of India, under project no.12-R\&D-TFR-5.01-0500.

\section{Fiber contraction and Hopf invariant}\label{sec-fclhopf} One of the only known techniques to estimate Urysohn width of manifolds from below is the fiber contraction lemma, introduced by Gromov and Guth. In this section, we combine the conclusion of the fiber contraction lemma with certain topological invariants, to produce estimates on the Urysohn width of total spaces of circle bundles in terms of hypersphericity of the base manifolds. Throughout this section, $\mathbf{k}$ will denote an arbitrary commutative unital ring. The reader may wish to specialize to $\mathbf{k} = \mathbf{Z}, \mathbf{Z}_2$. 

\subsection{Fiber contraction lemma} The following result was introduced by Gromov in \cite{gromovwidth}. We quote a generalisation due to Guth in \cite[Lemma 5.2]{guthsurface}.

\begin{lemma}[Fiber contraction lemma]\label{lem-fcl}
    Let $X$ be metric space, $K$ be a simplicial complex, and $M$ be a Riemannian manifold with convexity radius at least $R$.  Suppose we have 
    \begin{enumerate}
        \item A map $f:X \to K$ such that the diameter of any fiber of $f$ is at most $R$.
        \item  A $1$-Lipschitz map $p: X \to M$. 
    \end{enumerate}
Then there is a map $g: K \to M$, such that $p$ is homotopic to $g \circ f$. 
\end{lemma}

The following corollary will be used frequently in this article.

\begin{corollary} \label{cor-fcl}
    Let $p: X \to M$ be a $1$-Lipschitz map from a metric space $X$ to a Riemannian manifold $M$ of convexity radius at least $R > 0$. If the induced map $p_* : H_d(X; \mathbf{k}) \to H_d(M; \mathbf{k})$ is non-zero, then $\uw{d-1}(X) \geq R$.\end{corollary}
\begin{proof}
    Suppose otherwise that $\uw{d-1}(X) < R$. Then, by definition of Urysohn width, there is a continuous map $f: X \to K$ to a simplicial complex $K$ with $\dim K \leq d-1$ such that the diameter of every fiber of $f$ is at most $R$. By Lemma \ref{lem-fcl}, $p$ is homotopic to $g \circ f$ for some map $g: K \to M$. Therefore, $p_* = (g \circ f)_* = g_* \circ f_* : H_{d}(X; \mathbf{k}) \to H_d(M; \mathbf{k})$. But $f_* \equiv 0$, since $H_d(K; \mathbf{k})=0$ due to $\dim K \leq d-1$. Therefore, $g_* \circ f_* \equiv 0$. However, $p_*$ is not the zero map by assumption. Contradiction. Hence, $\uw{d-1}(X) \geq R$.     
\end{proof}

For a closed manifold $M$ of dimension $d$ and a map $f : M \to S^d$, we shall say $f$ has non-zero $\mathbf{k}$-degree if $f_* : H_d(M; \mathbf{k}) \to H_d(S^d; \mathbf{k}) \cong \mathbf{k}$ is non-zero, i.e. $f_*[M] \neq 0 \in \mathbf{k}$. The following definition is a generalization of the notion of hypersphericity \cite[Definition A]{gldirac} to non-integer coefficients. 

\begin{definition}[$\mathbf{k}$-hypersphericity] \label{def-hyp}
Let $M^d$ be a closed Riemannian manifold. Then the hypersphericity of $M$ is the supremal radius $R > 0$, such that $M^d$ admits a $1$-Lipschitz map of non-zero $\mathbf{k}$-degree onto a sphere $S^d(R)$ of radius $R$. We shall denote it by $\mathrm{HS}(M; \mathbf{k})$.
\end{definition}

We record the following corollary of Lemma \ref{cor-fcl}, due to Gromov \cite[Proposition $\mathrm{F}_1$]{gromovwidth}.

\begin{corollary}\label{cor-widthhyps}
    Let $M^d$ be a Riemannian manifold. Then $\uw{d-1}(M) \geq 1/2 \cdot \mathrm{HS}(M; \mathbf{k})$.
\end{corollary}
\begin{proof}
    Follows from Corollary \ref{cor-fcl} together with the observation that a sphere of radius $R > 0$ has convexity radius $R/2$.
\end{proof}

\subsection{Generalized Hopf invariant} In this section we prove an improvement of Corollary \ref{cor-fcl} by relaxing the hypothesis to allow $1$-Lipschitz maps to large Riemannian manifolds that are zero on top-dimensional homology of the target. To this end, we introduce a notion of mod-$2$ Hopf invariant for maps to spheres. This definition was inspired by \cite[Definition 2.1]{linmme}. First, we begin with some preliminaries. 

Let $X$ be a CW-complex and $f : X \to S^d$ is a continuous map. Let $C_f := CX \cup_f S^d$ denote the mapping cone of $f$. Suppose the induced map on cohomology $f^* : H^d(S^d; \mathbf{k}) \to H^d(X; \mathbf{k})$ is zero. Then, from the cohomology long exact sequence for the pair $(C_f, S^d)$, it follows that for degree $i>0$,
$$H^i(C_f; \mathbf{k}) \cong H^{i-1}(X; \mathbf{k}) \oplus H^i(S^d; \mathbf{k}).$$
Therefore, we have 
    \[ 
    H^i(C_f; \mathbf{k}) \cong
    \begin{cases}
        H^{i-1}(X; \mathbf{k}), & i \neq d\\
        H^{d-1}(X; \mathbf{k}) \oplus H^d(S^d; \mathbf{k}), & i  = d
    \end{cases}
    \]
Let $\eta_f := (0, 1) \in H^d(C_f; \mathbf{k})$ where $1 \in \mathbf{k} \cong H^d(S^d; \mathbf{k})$ is the multiplicative unit.

\begin{definition}\label{def-hopf}
    Let $X$ be a metric space, $f : X \to S^d$ be a continuous map such that $f$ induces the zero map in $\mathbf{Z}_2$-homology. By universal coefficient theorem, $$f^* : \mathbf{Z}_2 \cong H^d(S^d; \mathbf{Z}_2) \to H^d(X; \mathbf{Z}_2),$$
    is zero. Let $C_f$ denote the mapping cone of $f$, and $\eta_f \in H^d(C_f; \mathbf{Z}_2)$ be as defined above. We define the \emph{generalized mod-2 Hopf invariant} of $f$ as the Steenrod square of $\eta_f$:
    $$\mathrm{hop}(f) := \mathrm{Sq}^2(\eta_f) \in H^{d+2}(C_f; \mathbf{Z}_2) \cong  H^{d+1}(X; \mathbf{Z}_2)$$
\end{definition}

\begin{remark}
    If $f$ is nullhomotopic, then $\mathrm{hop}(f)=0$. This is because $C_f \simeq \Sigma X \vee S^d$ and therefore $\mathrm{Sq}^2(\eta_f) \in H^{d+2}(S^d;\ZZ_2) = 0$.
\end{remark}

The following result is an improvement of Corollary \ref{cor-fcl} with mod 2 coefficients.

\begin{corollary} \label{cor-fcl2}
    Let $X$ be a metric space. Let $p : X \to M^d$ be a $1$-Lipschitz map to a Riemannian manifold $M^d$ of convexity radius at least $R > 0$. 
    Let $q : M^d \to S^d$ be a map with non-zero $\mathbf{Z}_2$-degree. Let $f := q \circ p : X \to S^d$. Suppose, either
    \begin{enumerate}
        \item $f_* : H_d(X; \mathbf{Z}_2) \to H_d(S^d; \mathbf{Z}_2)$ is nonzero, or
        \item $f_*$ is zero but $\mathrm{hop}(f) \neq 0$.
    \end{enumerate}
    Then $\mathrm{UW}_{d-1}(X) \geq R$.
\end{corollary}
\begin{proof}
    If $f_*$ is nonzero, then $p_* : H_d(X; \mathbf{Z}_2) \to H_d(M; \mathbf{Z}_2)$ is also nonzero. Thus, the result follows from Corollary \ref{cor-fcl}. Therefore, suppose $f_*$ is zero and $\mathrm{hop}(f) \neq 0$. Assume $\mathrm{UW}_{d-1}(X) < R$. Then, there exists a continuous map $g : X \to K$ to some simplicial complex $K$ with $\dim K \leq d - 1$. By Lemma \ref{lem-fcl}, we have $p \simeq h \circ g$ for some map $h : K \to M$. Therefore, $f \simeq q \circ h \circ g$. By the cellular approximation theorem, $q \circ h : K \to S^d$ is nullhomotopic as $\dim K \leq d-1$. Thus, $f$ is nullhomotopic as well. This leads to a contradiction, as $\mathrm{hop}(f) \neq 0$.
    Thus, $\mathrm{UW}_{d-1}(X) \geq R$.
\end{proof}

\subsection{Hopf invariant for circle bundles} In this section, we show that circle bundles with non-zero Euler class mod $2$, also have non-zero Hopf invariant mod $2$, provided that the base manifold admits a $\mathrm{Pin}^{-}$ structure. The arguments in this section are a synthesis of the arguments in the proof of \cite[Theorem 2.1]{bol06} (see also, \cite[Lemma 3.2]{pscbundle}).

\begin{definition}
    Let $M$ be a smooth manifold. Let $w_1, w_2$ denote the first and second Stiefel-Whitney classes of the tangent bundle of $M$, respectively. We say $M$ admits a $\mathrm{Pin}^-$ structure, if the second Wu class $v_2 := w_2 + w_1^2 \in H^2(M; \mathbf{Z}_2)$ is zero.
\end{definition}
\begin{remark} \label{rem-sqvanish} 
    {Let $M^d$ be a $d$-dimensional manifold. By Wu's formula, $\sq2 : H^{d-2}(M;\ZZ_2) \to H^d(M; \ZZ_2)$ is given by $\sq2(x) = v_2(M) \smile x$. Thus, if $M$ admits a $\mathrm{Pin}^-$ structure, then $\sq2 = 0$ on $H^{d-2}(M; \ZZ_2)$.}
\end{remark}
A large class of manifolds that admit $\mathrm{Pin}^-$ structures are provided by hypersurfaces in spin manifolds.
\begin{lemma}\label{lem-pinhyp}
    Let $N$ be a spin manifold, and $M \subset N$ be a (not necessarily orientable) hypersurface. Then $M$ admits a $\mathrm{Pin}^-$ structure. 
\end{lemma}
\begin{proof}
    Recall the second Wu class of $M$, given by $v_2(TM) = w_2(TM) + w_1(TM)^2$.  We will show that $v_2(TM)=0$. Consider the inclusion map $i : M \to N$. Since $N$ is spin, we have that $w_1(TN)=0$ and $w_2(TN) = 0$. Therefore, the pullbacks $i^*w_2$ and $i^*w_1$ are also zero. Since $TN|_{M} \cong TM \oplus NM$, and $M$ is a hypersurface, the normal bundle $NM$ has real rank $1$. Therefore, $w_2(NM)=0$. Combining these facts and applying the Whitney formula, we get 
    \begin{align*}
        0 &= i^*w_1(TN) = w_1(TN|_{M}) = w_1(T{M}) + w_1(NM) \implies w_1(TM)=w_1(NM) \\
    0 &= i^*w_2(TN) = w_2(TN|_{M}) = w_2(TM) + w_1(TM) \smile w_1(NM) 
    \end{align*}
    Combining these two we get 
    $v_2(TM) = w_2(TM) + w_1^2(TM) = 0,$
    as desired.
\end{proof}
We now come to the main result of this section. We begin with an elementary lemma.
\begin{lemma}\label{lem-conethom}
    Let $p : E \to M$ be a circle bundle. Let $\pi : \mathbb{D}(E) \to M$ be the disk bundle corresponding to $E$, such that $E = \mathbb{S}(E)$ is the unit circle bundle, and $\pi|_E = p$. Let $\mathrm{Th}(E) = \mathbb{D}(E)/\mathbb{S}(E)$ denote the Thom space of $\pi$. Let $C_p$ denote the mapping cone of $p : E \to M$. Then $C_p \cong \mathrm{Th}(E)$.
\end{lemma}
\begin{proof}
    Note that the mapping cone $C_p$ is obtained from the mapping cylinder $$\mathrm{Cyl}_p = E \times [0, 1] \bigsqcup_{(x, 0) \sim p(x)} M$$
    by quotienting $E \times \{1\}$. For a point $z = (m, v) \in \mathbb{D}(E)$ with $\pi(z) = m$, we denote $r \cdot z = (m, r \cdot v) \in \mathbb{D}(E)$ for $r \in [0, 1]$. Let 
    $\Phi : \mathrm{Cyl}_p \to \mathbb{D}(E),$
    be defined by $\phi(x, r) = r \cdot x$ on $E \times [0, 1] \subset \mathrm{Cyl}_p$ and $\phi(m) = p(m)$ on $M \subset \mathrm{Cyl}_p$. Then $\phi$ establishes a homeomorphism such that $\phi(E \times \{1\}) = \mathbb{S}(E)$. Therefore, we get $$C_p = \mathrm{Cyl}_p/(E \times \{1\}) \cong \mathbb{D}(E)/\mathbb{S}(E) = \mathrm{Th}(E),$$
    as required.
\end{proof}

\begin{prop}\label{prop-hopfcircle}
    Let $M^{d}$ be a closed manifold that admits a $\mathrm{Pin}^-$ structure. Let $p : E^{d+1} \to M^{d}$ be a principal circle bundle over $M$, and $q : M^{d} \to S^{d}$ be a map with non-zero $\mathbf{Z}_2$-degree. Let $f := q \circ p : E^{d+1} \to S^{d}$. If $f^* : H^d(S^d; \mathbf{Z}_2) \to H^d(E; \mathbf{Z}_2)$ is zero, then $\mathrm{hop}(f) \neq 0$.
\end{prop}
\begin{proof}
    Note that $q$ defines a map $\overline{q} : C_p \to C_f$. By naturality of Steenrod squares, $$\overline{q}^* \mathrm{Sq}^2(\eta_f) = \mathrm{Sq}^2(\overline{q}^* \eta_f).$$
    From Lemma \ref{lem-conethom}, we have $C_p \cong \mathrm{Th}(E)$ is the Thom space of the disk bundle $\pi : \mathbb{D}(E) \to M$ corresponding to the circle bundle $p : E \to M$. Let $\alpha := \overline{q}^* \eta_f \in H^d(\mathrm{Th}(E); \mathbf{Z}_2)$. It suffices to show $\mathrm{Sq}^2(\alpha) \neq 0$. Let us denote $u \in H^2(\mathrm{Th}(E); \mathbf{Z}_2)$ to be the Thom class. Let
 \begin{gather*}
    \Phi : H^i(M; \mathbf{Z}_2) \to H^{i+2}(\mathrm{Th}(E); \mathbf{Z}_2) \cong {H^{i+2}(\mathbb{D}(E), \mathbb{S}(E); \mathbf{Z}_2)},\\ 
    \Phi(x) = \pi^*x \smile u,
    \end{gather*}
    denote the Thom isomorphism. Let $\xi \in H^{d-2}(M; \mathbf{Z}_2)$ be such that $\Phi(\xi) = \alpha$.
    
    Let $s : M \to \mathrm{Th}(E)$ denote the zero section. Let $w_i$, $i = 1, 2$, denote the first and second Stiefel-Whitney classes of the bundle $\pi : \mathbb{D}(E) \to M$. Note that $s^*u = w_2$. Thus,
    \begin{align*}
        \xi \smile w_2 = s^*\pi^*\xi \smile s^*u = s^*(\pi^*\xi \smile u) = s^*\Phi(\xi) = s^*\alpha = s^*\overline{q}^*\eta_f 
        &= (\overline{q} \circ s)^* \eta_f \\
        &= q^* \eta_f \\ &\neq 0,
    \end{align*}
    since $q : M \to S^d$ has nonzero $\mathbf{Z}_2$-degree. Now, since $M$ admits a $\mathrm{Pin}^-$ structure, we have $\sq2 = 0$ on $H^*(M; \ZZ_2)$ by Remark \ref{rem-sqvanish}. In particular, we have $\sq2(\xi) =0$. This along with Cartan's formula gives, \begin{align*}\mathrm{Sq}^2(\alpha) = \mathrm{Sq}^2(\Phi(\xi)) &= \mathrm{Sq}^2(\pi^*\xi \smile u)\\ &= \pi^*\mathrm{Sq}^2(\xi) \smile u + \pi^*\mathrm{Sq}^1(\xi) \smile \mathrm{Sq}^1(u) + \pi^*\xi \smile \mathrm{Sq}^2(u) \\
    &= \pi^*\mathrm{Sq}^1(\xi) \smile \mathrm{Sq}^1(u) + \pi^*\xi \smile \mathrm{Sq}^2(u)
    \end{align*}
    By the Thom-Wu formula, $\mathrm{Sq}^1(u) = \Phi(w_1)$ and $\mathrm{Sq}^2(u) = \Phi(w_2)$. Since $w_1 = 0$ by orientability of the circle bundle, $\mathrm{Sq}^1(u) = 0$. Therefore, the above reduces to $$\mathrm{Sq}^2(\alpha) = \pi^*\xi \smile \Phi(w_2) = \pi^*\xi \smile \pi^* w_2 \smile u = \pi^*(\xi \smile w_2) \smile u = \Phi(\xi \smile w_2)$$
    But we saw earlier that $\xi \smile w_2 \neq 0$. Since $\Phi$ is an isomorphism, this implies $\sq 2 (\alpha) \neq 0$. Therefore, $\mathrm{hop}(f) \neq 0$, as desired.
\end{proof}
\begin{remark}
    We note that Proposition \ref{prop-hopfcircle} is not true if the $\mathrm{Pin}^-$ hypothesis on $M$ is dropped. Indeed, consider the five dimensional Hopf bundle $p : S^5 \to \mathbf{CP}^2$ and let $q : \mathbf{CP}^2 \to S^4$ be the degree $1$ map obtained from crushing $\mathbf{CP}^1 \subset \mathbf{CP}^2$. Then, $q \circ p$ is the attaching map of the $6$-cell in $\mathbf{CP}^3/\mathbf{CP}^1 \simeq S^4 \vee S^6$. Therefore, it is null-homotopic, and in particular has trivial mod $2$ Hopf invariant.
\end{remark}

\subsection{Width of circle bundles} The main result of this section is that codimension $2$ Urysohn width of total spaces of circle bundles over $\mathrm{Pin}^-$ manifolds equipped with $U(1)$-invariant metrics, is bounded from below by the hypersphericity of the base manifold. This can be considered a codimension-$2$ variant of Corollary \ref{cor-widthhyps}. We start with the following terminology.

\begin{definition}[Adapted metric] \label{def-adapted}
    Let $(M,g_M)$ be a Riemannian manifold and $p : E \to M$ be a principal circle bundle. Let $r > 0$ be any real number, and $\theta \in \Omega^1(E; \mathfrak{u}(1))$ be any $U(1)$-invariant connection $1$-form on $E$. The \emph{adapted metric} on $E$ with parameters $(r, \theta)$ will be defined as the Riemannian metric $g_E = p^* g_M + r^2~\theta \otimes\theta$. 
\end{definition}

Note that $p : (E, g_E) \to (M, g_M)$ is a Riemannian submersion by definition. We now come to the main result of this section.

\begin{theorem}\label{cor-fcl3}
  Let $M^{d-1}$ be a closed Riemannian manifold which admits a $\mathrm{Pin}^-$ structure. Let $p : E \to M$ be a principal circle bundle over $M$. We equip $E$ with any adapted metric. Then, $\uw{d-2}(E) \geq 1/2 \cdot \mathrm{HS}(M; \mathbf{Z}_2)$.
\end{theorem}

\begin{proof}
    Suppose there exists a $1$-Lipschitz map $q : M \to S^{d-1}(2R)$ to a sphere of radius $2R$, such that $q$ has nonzero $\mathbf{Z}_2$-degree. Let $f := q \circ p : E^{d} \to S^{d-1}(2R)$. As $p$ is a Riemannian submersion, it is $1$-Lipschitz. Therefore, the composition $f$ is also $1$-Lipschitz. 
    
    Note that the convexity radius of $S^{d-1}(2R)$ is at least $R$. If the induced homomorphism $f^* : H^{d-1}(S^{d-1}; \mathbf{Z}_2) \to H^{d-1}(E; \mathbf{Z}_2)$ is zero, then by Proposition \ref{prop-hopfcircle}, we get $\mathrm{hop}(f) \neq 0$, since $q$ has non-zero $\mathbf{Z}_2$-degree. Thus, either $f^* \neq 0$ or $f^* = 0$ and $\mathrm{hop}(f) \neq 0$. In either case, by Corollary \ref{cor-fcl2}, we obtain $\uw{d-2}(E) \geq R$. Taking a supremum over $R$, we obtain $\mathrm{UW}_{d-2}(E) \geq 1/2 \cdot \mathrm{HS}(M; \mathbf{Z}_2)$, as required.
\end{proof}

The results in the following section are not used in the rest of the article. However, it provides an instructive contrast to the results discussed {above}.

\subsection{Urysohn width and Cheeger-Gromov collapsing} In the following proposition, we revisit a particular class of examples of `small' circle bundles over `large' manifolds provided in our earlier work \cite{pscbundle} which demonstrate that Urysohn width is not continuous under Cheeger-Gromov collapsing limits. As a consequence, we shall obtain that the statement of Theorem \ref{cor-fcl3} does not hold if the $\mathrm{Pin}^-$ hypothesis is removed, or if the $\mathbf{Z}_2$-hypersphericity hypothesis is replaced by $\mathbf{Z}$-hypersphericity. 

\begin{prop}\label{prop-pscbundles}
    Urysohn width is not continuous with respect to Cheeger-Gromov collapsing limits. In particular, it is not continuous under Gromov-Hausdorff convergence.
\end{prop}

\begin{proof}
      Let us recall the construction from \cite{pscbundle}. For simplicity, consider the symplectic $6$-manifold  $(T^4 \times S^2, \omega)$ obtained by taking the product of a symplectic $4$-torus and the $2$-sphere with the standard area form. Let $Z^4 \subset (T^4 \times S^2, \omega)$ be a Donaldson divisor with twisting parameter $k \gg 1$, i.e. $\mathrm{PD}[Z] = k[\omega]$. Let $\pi : E \to Z$ denote the unit normal bundle of $Z$ in $T^4 \times S^2$. By Donaldson's symplectic Lefschetz hyperplane theorem, 
      $$\pi_1(Z) \cong \pi_1(T^4 \times S^2) \cong \mathbf{Z}^4.$$
      Let us equip $E$ with an adapted metric $g_{E} = \pi^* g_Z + \theta \otimes \theta$ (Definition \ref{def-adapted}), with fiber circles of unit radius. From the proof of \cite[Theorem A, Part (2)]{pscbundle} as well as \cite[Lemma 3.1]{pscbundle}, we obtain a map $f : E \to K$ to a simplicial complex with $\dim K \leq 3$ such that:
      \begin{enumerate}
          \item $f_* : \pi_1(E) \to \pi_1(K)$ is an isomorphism,
          \item Let $H \leq \pi_1(E) \cong \pi_1(K)$ be any subgroup. Let $\widetilde{E}, \widetilde{K}$ denote the covering spaces of $E, K$ corresponding to $H$, respectively. There exists a constant $C = C(g_E, K, f)$ independent of $H$, such that the diameter of the fibers of the lift $\widetilde{f} : \widetilde{E} \to \widetilde{K}$ are uniformly bounded by $C$. In particular, $\mathrm{UW}_3(\widetilde{E}, \widetilde{g}_E) \leq C$.
      \end{enumerate}
      For any $n > 0$, let $p_n : Z_n \to Z$ be the cover corresponding to the subgroup $(n\mathbf{Z})^4 \leq \mathbf{Z}^4$. Let $\pi_n : E_n \to Z_n$ denote the pullback of $E$ over this cover, equipped with the pullback of the adapted metric $g_n := p_n^*g_E$. By \cite[Proposition 5]{pscbundle} (see also, \cite[Theorem A, Part (4)]{pscbundle}), $Z_n$ admits a positive degree $1$-Lipschitz map to the flat $4$-torus $\mathbf{R}^4/(n\mathbf{Z})^4$. Since the target has injectivity radius of the order of $n$, $\mathrm{UW}_3(Z_n; \mathbf{Z}) \gtrsim n$ by Corollary \ref{cor-fcl}. 
      
      We now fix $n \gg C$. Let $g_{E, \varepsilon} = \pi^* g_Z + \varepsilon^2 \theta \otimes \theta$ be the adapted metric on $E$ associated with the same auxiliary connection as $g_E$, but with fiber radius $\varepsilon < 1$. Let $g_{n, \varepsilon} := p_n^* g_{E, \varepsilon}$. Then $g_{n, \varepsilon} < {g_n}$, therefore $\mathrm{UW}_3(E_n, g_{n, \varepsilon}) < \mathrm{UW}(E_n, g_n) \leq C$ as well. Thus, even though there is a Cheeger-Gromov collapse $(E_n, g_{n, \varepsilon}) \xrightarrow{\mathrm{GH}} Z_n$, we still have
      \begin{equation*}
      \limsup_{\varepsilon \to 0} \mathrm{UW}_3(E_n, g_{n, \varepsilon}) \leq C \ll n \lesssim \mathrm{UW}_3(Z_n).
      \end{equation*}
      This proves the result.
\end{proof}

\begin{remark}\label{rem-pscbundles}
We observe here that the example above do not satisfy the conclusion of Theorem \ref{cor-fcl3}. Note that $\mathrm{HS}(Z_n; \mathbf{Z}) \gtrsim n$, as the positive degree $1$-Lipschitz map to the flat torus $\mathbf{R}^4/(n\mathbf{Z})^4$ can be composed with a degree one $1$-Lipschitz map to a sphere of radius comparable with $n$. Recall the twisting parameter $k$ of the Donaldson divisor $Z \subset (T^4 \times S^2, \omega)$. 

Now, suppose the twisting parameter $k$ is odd. Then, $w_2(Z_n) \neq 0$ by \cite[Proof of (4), pg.~14]{pscbundle}. Since $Z_n$ is orientable, this shows that $Z_n$ does not admit a $\mathrm{Pin}^{-}$ structure. 

On the other hand, suppose $k$ even. From the proof of \cite[Proposition 5]{pscbundle}, we know that the degree of the map to the torus is divisible by $k$, hence is even. Therefore, even though $\mathrm{HS}(Z_n; \mathbf{Z}) \gtrsim n$, there is apriori no lower bound on $\mathrm{HS}(Z_n; \mathbf{Z}_2)$. In fact, aposterori, Theorem \ref{cor-fcl3} implies $\mathrm{HS}(Z_n; \mathbf{Z}_2) < C$ for some constant $C > 0$ independent of $n$.
\end{remark}

\section{Singular ruling and macroscopic scalar curvature}\label{sec-strat} In this section, we introduce sufficient conditions on the base manifolds of circle bundles that ensure that the total spaces admit metrics with arbitrarily large macroscopic scalar curvature. We begin with the following rather soft notion, inspired by a comment of Gromov \cite[Section 1.2, pg.~10]{grom}. 

\begin{definition}\label{def-ruling}
    Let $M$ be a Riemannian manifold, $\lambda > 0$ be a positive constant. A \emph{pre-ruling} of $M$ is a collection of smoothly embedded $2$-spheres $\mathfrak{S} = \{\Sigma_\alpha \subset M\}$, such that for all $p \in M$ there exists at least one sphere $\Sigma_\alpha \in \mathfrak{S}$ passing through $p$. 
    
    We call $\Sigma_\alpha \in \mathfrak{S}$ the \emph{ruling spheres}. Note that we do not demand continuity of the family of ruling spheres $\mathfrak{S}$ in any sense. Furthermore, ruling spheres are allowed to intersect each other. We also make the following definitions:
    \begin{enumerate}
    \item A pre-ruling $\mathfrak{S}$ of $M$ is \emph{$\lambda$-small} if for all $\Sigma \in \mathfrak{S}$, $\diam(\Sigma) < \lambda$.
    \item A pre-ruling $\mathfrak{S}$ of $M$ is a \emph{ruling} if there exists an integral cohomology class $\omega \in H^2(M; \mathbf{Z})$ such that $\langle \omega, [\Sigma] \rangle = \pm 1$ for all $\Sigma \in \mathfrak{S}$. In this case, we say $\omega$ \emph{tames} $\mathfrak{S}$. 
    \end{enumerate} 
\end{definition}

\begin{definition}\label{def-assoccirc}
Let $M$ be closed Riemannian manifold with a $\lambda$-small ruling $\mathfrak{S}$. Let $\omega \in H^2(M; \mathbf{Z})$ denote a taming class for the ruling $\mathfrak{S}$. Consider the principal circle bundle
$$S^1\hookrightarrow E \stackrel{\pi}{\to} M$$
with $c_1(E) = \omega$. We say $\pi : E \to M$ is a circle bundle \emph{associated} to the ruling $\mathfrak{S}$.
\end{definition}

The key result of this section is that the total space $E$ equipped with a $U(1)$-invariant adapted metric (cf. Definition \ref{def-adapted}) satisfies the hypothesis of Conjecture \ref{conj-abgtrue}, provided the circle fibers of $\pi : E \to M$ have sufficiently small radius. 

\begin{theorem}\label{thm-circbdlmscal}
    Let $(M^d, g_M)$ be a closed Riemannian manifold with a $3/2$-small ruling $\mathfrak{S}$. Let $\pi : E \to M$ be the circle bundle associated with $\mathfrak{S}$. We equip $E$ with an adapted metric $$g_E = \pi^* g_M + \left ( \frac\rho\pi \right )^2 \theta \otimes \theta,$$
    for some auxiliary choice of connection $\theta \in \Omega^1(E; \mathfrak{u}(1))$. Given any $\sigma > 0$, there exists $\rho > 0$ such that $(E, g_E)$ has macroscopic scalar curvature $\msc \geq \sigma$ (cf. Definition \ref{def-msc}).
\end{theorem}

\begin{proof}
    Select a point $x \in E$ and let $y := \pi(x) \in M$. Let $p : \widetilde{B}_2(x) \to B_2(x)$ denote the universal cover of the ball $B_2(x)$ of radius $2$ centered at $x$. Select $\widetilde{x} \in \widetilde{B}_2(x)$ such that $p(\widetilde{x}) = x$. Let us equip $\widetilde{B}_2(x)$ with the pulled back metric. Let $B_1(\widetilde{x}) \subset \widetilde{B}_2(x)$ denote the unit ball centered at $\widetilde{x}$. We shall find a uniform $\rho > 0$ independent of $x$ such that   $$\mathrm{vol}(B_1(\widetilde{x})) < V^{d+1}_\sigma(1),$$ where $V^{d+1}_\sigma(1)$ denotes the volume of a unit ball in a simply connected $(d+1)$-dimensional Riemannian space form of constant scalar curvature $\sigma$. 

    \begin{enumerate}[topsep=1ex,itemsep=1ex,partopsep=0.5ex,parsep=0.5ex,label=\textbf{Step \arabic*.}, wide=0pt]
    \item Let $B_{3/2}(y)$ denote the ball of radius $3/2$ centered at $y$. First, we show that the following chain of inclusion holds:  
    $$B_1(x) \subset \pi^{-1} B_{3/2}(y) \subset B_2(x).$$
    The first inclusion follows since $\pi$ is a Riemannian submersion and therefore is 1-Lipschitz. For the second inequality, select a point $z \in \pi^{-1} B_{3/2}(y)$. Choose a minimal geodesic segment $\gamma \subset M$ from $y$ to $\pi(z)$. Let $\widetilde{\gamma} \subset E$ be the horizontal lift of $\gamma$ with initial point $x$. Note that the length of $\widetilde{\gamma}$ is the same as that of $\gamma$. Therefore, $\ell(\widetilde{\gamma}) = \ell(\gamma) \leq 3/2$. Let $w$ be the endpoint of $\widetilde{\gamma}$. Note that $\pi(w) = \pi(z)$. Therefore, $$\mathrm{dist}(x, z) \leq \mathrm{dist}(x, w) + \mathrm{dist}(w, z) \leq \ell(\widetilde{\gamma}) + \mathrm{diam}(\pi^{-1}(\pi(z))) \leq 3/2 + \rho.$$ Choosing $\rho < 1/2$, we may thus ensure $\mathrm{dist}(x, z) < 2$ for all $z \in \pi^{-1} B_{3/2}(y)$. Hence, the second inclusion follows. Let $i : \pi^{-1} B_{3/2}(y) \hookrightarrow B_2(x)$ denote the inclusion map.
    
    \item Next, we show that $\pi : \pi^{-1} B_{3/2}(y)  \to B_{3/2}(y)$ induces an isomorphism on fundamental groups. As this is a circle bundle, using the homotopy long exact sequence for fiber bundles we observe that it suffices to check the inclusion map of the fiber $\pi^{-1}(y) \hookrightarrow \pi^{-1} B_{3/2}(y)$ is $\pi_1$-null. To this end, let $\Sigma_y \in \mathfrak{S}$ be a ruling sphere passing through $y$. As $\mathrm{diam}(\Sigma_y) < 3/2$, we have $\Sigma_y \subset B_{3/2}(y)$. Consequently, the inclusion map of the fiber factors as
    $$\pi^{-1}(y) \hookrightarrow \pi^{-1}(\Sigma_y) \hookrightarrow \pi^{-1} B_{3/2}(y).$$
    Since $\pi : E \to M$ is a fiber bundle associated with the ruling $\mathfrak{S}$, we have $\langle c_1(E), [\Sigma_y]\rangle = \pm 1$. Therefore, $\pi : \pi^{-1}(\Sigma_y) \to \Sigma_y$ is a circle bundle over the 2-sphere with Euler class $\pm 1$. In other words, it is the Hopf fibration. In particular, $\pi^{-1}(\Sigma_y) \cong S^3$ is simply connected. This proves the inclusion map of the fiber is $\pi_1$-null.

    \item Let $U$ denote the connected component of $(\pi \circ p)^{-1} B_{3/2}(y) \subset \widetilde{B}_2(x)$ containing $\widetilde{x}$. Then, 
    $$p|_U : U \to \pi^{-1} B_{3/2}(y)$$
    defines a covering map corresponding to the subgroup $\ker i_* \leq \pi_1(\pi^{-1} B_{3/2}(y), x)$. By {\bf Step 2}, $\pi : \pi^{-1} B_{3/2}(y) \to B_{3/2}(y)$ induces an isomorphism on fundamental groups. Thus, $\ker(i_*)$ is isomorphic to the subgroup $H := \pi_*(\ker i_*) \leq \pi_1(B_{3/2}(y), y)$. Let $q : V \to B_{3/2}(y)$ be the cover corresponding to $H$. Let $W$ denote the fiber product of the covering map $q$ and the circle bundle $\pi$ over $B_{3/2}(x)$, illustrated in the following fiber square:
    \begin{center}
    \begin{tikzcd}[row sep=9ex, column sep=5ex]
    W \arrow[d, "g"'] \arrow[r, "f"]     & V \arrow[d, "q"'] \\
    \pi^{-1} B_{3/2}(x) \arrow[r, "\pi"] & B_{3/2}(x)       
    \end{tikzcd}
    \end{center}
    Since $f$ is the pullback of the circle bundle $\pi$ over $V$, it is a circle bundle. Likewise, since $g$ is the pullback of the covering map $q$ over $\pi^{-1} B_{3/2}(x)$, it is a covering map. Therefore, at the level of the fundamental group, $g_*$ and $q_*$ are injective. Since we have already established $\pi_*$ is an isomorphism, this implies $f_*$ is injective as well. Since $f$ is a fiber bundle with connected fibers, $f_*$ is surjective. Thus, $f$ induces an isomorphism $f_* : \pi_1(W) \to \pi_1(V)$. Consequently, 
    $$(\pi_* \circ g_*)\pi_1(W) = q_*\pi_1(V) = H.$$
    Hence, $g_* \pi_1(W) = (\pi_*)^{-1}H = \ker i_*$. Therefore, $W \cong U$ as covering spaces over $\pi^{-1} B_{3/2}(y)$. Furthermore, under this isomorphism, the metric on $U$ induced from the inclusion $U \subset \widetilde{B}_2(x)$ agrees with the adapted metric on $W$ given by 
    $q^*g_M + (\rho/\pi)^2 \cdot q^*\theta \otimes q^*\theta$,
    obtained from pulling back the adapted metric on $\pi^{-1} B_{3/2}(y) \subset E$. 

    \item Since $p : B_2(\widetilde{x}) \to B_2(x)$ is a covering map, it is $1$-Lipschitz. Therefore, $B_1(\widetilde{x}) \subset p^{-1} B_1(x)$. By {\bf Step 1}, we have $B_1(x) \subset \pi^{-1} B_{3/2}(y)$. Therefore, $B_1(\widetilde{x}) \subset (\pi \circ p)^{-1} B_{3/2}(y)$. As $B_1(\widetilde{x})$ is connected, it is contained in the connected component $U \subset (\pi \circ p)^{-1} B_{3/2}(y)$. Further, note that $\pi$ is $1$-Lipschitz, as it is a Riemannian submersion. Therefore, $\pi \circ p : U \to B_{3/2}(y)$ is also $1$-Lipschitz. Consequently, $$B_1(\widetilde{x}) \subset (\pi \circ p)^{-1} B_1(y) \subset (\pi \circ p)^{-1} B_{3/2}(y).$$ 
    Since $f : U \cong W \to V$ is a Riemannian submersion as well, it is $1$-Lipschitz. Therefore, $f(B_1(\widetilde{x})) \subset V$ is contained in the unit ball $B \subset V$ centered at $f(\widetilde{x})$. By the co-area formula applied to the circle bundle $f : W \to V$, we obtain
    $$\mathrm{vol}(B_1(\widetilde{x})) \leq 2\rho \cdot \mathrm{vol}(B),$$
    since the length of the circle fibers of $f$ is $2\rho$. The quantity $\mathrm{vol}(B)$ is roughly like the macroscopic scalar curvature of $M$ at $y$, with slightly different parameters. Since $M$ is a compact Riemannian manifold, there exists a constant $\kappa > 0$ such that the Ricci curvature of $M$ satisfies the lower bound $\mathrm{Ric} \geq -\kappa/d$. This bound holds for the ball $B_{3/2}(y) \subset M$ as well as for the cover $V$ of $B_{3/2}(y)$. By the Bishop-Gromov comparison inequality, $\mathrm{vol}(B) \leq V_{-\kappa}^d(1)$. 
    \end{enumerate}
    
    Therefore, let
    $\rho <  2^{-1} \min(1, V^{d+1}_\sigma(1)/V_{-\kappa}^d(1))$. From the discussion above, we obtain $\mathrm{vol}(B_1(\widetilde{x})) < V^{d+1}_\sigma(1)$. Hence, $\mathrm{mscal}(x) \geq \sigma$ for all $x \in E$, as desired.
\end{proof}

\section{Two constructions of base manifolds}\label{sec-construction} 

In this section, we give constructions of $\mathrm{Pin}^-$ manifolds which admit small rulings, yet have large hypersphericity radius. These will serve as the base manifolds for the circle bundles whose total spaces will be our counterexamples to Conjecture \ref{conj-abgtrue}.

\subsection{Corrugation of enlargeable manifolds} \label{sec-corrugration} 

The following notion is a variant of the notion of enlargeable manifolds introduced by Gromov and Lawson \cite{glpi1}, adapted for mod $2$ coefficients. See, \cite[Definition 4.1]{lz}.

\begin{definition}\label{def-z2en}
    A closed Riemannian manifold $N$ of dimension $d$ is called \emph{$\mathbf{Z}_2$-enlargeable} if for any $R> 0$, there exists a finite sheeted cover of $N$ with $\mathbf{Z}_2$-hypersphericity at least $R$. 
\end{definition}

\begin{exam}\label{eg-z2en} 
    In what follows, we summarize some properties of $\mathbf{Z}_2$-enlargeability.
    \begin{enumerate}
    \item The circle $S^1 = \mathbf{R}/\mathbf{Z}$ is a $\mathbf{Z}_2$-enlargeable manifold. Indeed, for any $n > 0$, the $n$-sheeted cover $\mathbf{R}/n\mathbf{Z}$ of $\mathbf{R}/\mathbf{Z}$ has hypersphericity $\sim n$.

    \item Products of $\mathbf{Z}_2$-enlargeable manifolds is also $\mathbf{Z}_2$-enlargeable. Therefore, the torus $T^d = S^1 \times \cdots \times S^1$ is $\mathbf{Z}_2$-enlargeable.

    \item Any manifold $M$ admitting a non-zero $\mathbf{Z}_2$-degree map to a $\mathbf{Z}_2$-enlargeable manifold, is also $\mathbf{Z}_2$-enlargeable. Therefore, if $M$ is a closed $d$-manifold possesses a non-zero $\mathbf{Z}_2$-degree map $M^d \to T^d$, then $M$ is $\mathbf{Z}_2$-enlargeable.

    \item $\mathbf{Z}_2$-enlargeable manifolds of dimension $d \leq 7$ do not admit positive scalar curvature metrics, see \cite[Theorem 4.2]{lz}.
    \end{enumerate}
\end{exam}

\begin{remark}
    Orientable $\mathbf{Z}$-enlargeable manifolds need not always be $\mathbf{Z}_2$-enlargeable. Indeed, consider the $4$-manifold $Z$ from Remark \ref{rem-pscbundles}, i.e., a Donaldson divisor of $(T^4 \times S^2, \omega)$ with $\mathrm{PD}[Z] = k[\omega]$. Then, Remark \ref{rem-pscbundles} shows $Z$ is $\mathbf{Z}$-enlargeable. Suppose that $Z$ was also $\mathbf{Z}_2$-enlargeable. Then, there is a sequence of covers $Z'_n \to Z$ such that $\mathrm{HS}(Z'_n; \mathbf{Z}_2) \geq n$. Note that $Z$ is orientable as it is symplectic. Suppose the twisting parameter $k$ is \emph{even}. Then, $Z$ is also spin (see, \cite[Proof of $(4)$, pg.~14]{pscbundle}). Therefore, $Z_n'$ is also spin, and in particular they admit $\mathrm{Pin}^-$ structures. The proof of Proposition \ref{prop-pscbundles} produces a circle bundle $E_n' \to Z_n'$ such that $\mathrm{UW}_3(E_n') \ll n/2 \leq 1/2~ \mathrm{HS}(Z_n'; \mathbf{Z}_2)$, for sufficiently large $n \gg 1$. This violates Theorem \ref{cor-fcl3}, giving a contradiction. Hence, $Z$ is not $\mathbf{Z}_2$-enlargeable.
\end{remark}

Let $N^d$ be a $\mathbf{Z}_2$-enlargeable manifold of dimension $d \geq 3$, that admits a $\mathrm{Pin}^-$ structure. Fix $\delta > 0$ be a positive constant such that $\delta \ll \mathrm{inj.rad.}(N)$. By definition of $\mathbf{Z}_2$-enlargeability, there exists a finite cover $N_n \to N$ such that $\mathrm{HS}(N_n; \mathbf{Z}_2) \geq n$. Select a maximal $\delta$-separated collection of points, i.e., a discrete subset $S \subset N$ such that 
\begin{enumerate}
    \item $S$ is $\delta$-dense, i.e. $N \subset \bigcup_{x \in S} B_\delta(x)$,
    \item $S$ is $\delta/2$-sparse, i.e. $\{B(x; \delta/2) : x \in S\}$ is a pairwise disjoint collection of subsets.
\end{enumerate}
Let $S_n \subset N_n$ be the preimage of $S \subset N$ by the covering map $N_n \to N$. Then $S_n \subset N_n$ is also maximal $\delta$-separated. Let $Q = S^2(1/2) \times S^{d-2}(1)$ denote the product of two round spheres of indicated radii and dimensions. Fix a point $q \in Q$. We define $M_n$ by gluing a copy of $Q\setminus B_{\delta/2}(q)$ to every boundary component of $N_n \setminus \bigcup_{x \in S_n} B_{\delta/2}(x)$. Smoothing the metric near the gluing region, we obtain a Riemannian metric on $M_n$. Note that, topologically, 
$$M_n \cong N_n \#^{|S_n|} (S^2 \times S^{d-2}).$$

\begin{prop}\label{prop-corrugation1}
    $M_n$ admits a $\mathrm{Pin}^-$ structure.
\end{prop}

\begin{proof}
    Since $N$ admits a $\mathrm{Pin}^-$ structure, its second Wu class vanishes. By naturality of the Wu class, we see that $N_n$ also admits a $\mathrm{Pin}^-$ structure. Note that $S^2 \times S^{d-2}$ is spin (hence also $\mathrm{Pin}^-$) and connect sum of $\mathrm{Pin}^-$ manifolds continues to be a $\mathrm{Pin}^-$ manifold. Therefore, $M_n \cong N_n \#^{|S_n|} (S^2 \times S^{d-2})$ admits a $\mathrm{Pin}^-$ structure.
\end{proof}

\begin{prop}\label{prop-corrugation2}
    $\mathrm{HS}(M_n; \mathbf{Z}_2) \gtrsim n$.
\end{prop}

\begin{proof}
    Since $\mathrm{HS}(N_n; \mathbf{Z}_2) \geq n$, there exists a $1$-Lipschitz map of non-zero $\mathbf{Z}_2$-degree from $N_n$ to a sphere of radius $\gtrsim n$. We compose it with the degree $1$ map $M_n \to N_n$ given by crushing the $S^2 \times S^{d-2}$ connect summands. The composition continues to be $1$-Lipschitz and of non-zero $\mathbf{Z}_2$-degree. Thus, $\mathrm{HS}(M_n; \mathbf{Z}_2) \geq n$.
\end{proof}

\begin{prop}\label{prop-corrugation3}
    $M_n$ admits a $3/2$-small ruling.
\end{prop}

\begin{proof}
    The manifold $Q = S^2(1/2) \times S^{d-2}(1)$ admits a $1$-small pre-ruling by the spheres $S^2(1/2) \times \{*\}$. Some of these spheres intersect the ball $B_{\delta/2}(q) \subset Q$. We isotope the spheres away from this ball by a small isotopy, to obtain a $(1+O(\delta))$-small pre-ruling on $Q \setminus B_{\delta/2}(q)$. Next, for any point $z \in N_n \setminus \bigcup_{x \in S_n} B_{\delta/2}(x)$, we consider a minimal geodesic ray $\gamma$ from $z$ to the nearest boundary component of $N_n \setminus \bigcup_{x \in S_n} B_{\delta/2}(x)$. Since $S_n$ is $\delta$-dense, length of $\gamma_z$ is at most $\delta$. Suppose the terminal point of $\gamma$ is contained in the ruling sphere $S^2 \cong \Sigma \subset Q \setminus B_{\delta/2}(q)$. Then $\Sigma \cup \gamma$ is a `sphere with a hair' containing $z$. {Consider a very thin normal thickening $N(\gamma) \cong [0, 1] \times D^2(\delta/100)$ of $\gamma \cong [0, 1]$, with $\{0\} \times D^2$ containing $z$ and $\{1\} \times D^2 \subset \Sigma$}. Let $\Sigma_z = \partial N(\gamma) ~\triangle~ \Sigma$ be the symmetric difference. Then $\Sigma_z \subset M_n$ is an embedded sphere containing $z$. Effectively, $\Sigma_z$ is obtained from $\Sigma \cup \gamma$ by replacing the `hair' by a thin `finger' protruding from $\Sigma$. Note that $\Sigma_z$ is homologous to $\Sigma \subset Q$, and $\diam(\Sigma_z) = 1 + O(\delta) < 3/2$.

    The collection of all the ruling spheres of $Q \setminus B_{\delta/2}(q)$ and the spheres of the form $\Sigma_z$ gives a pre-ruling $\mathfrak{S}_n$ of $M_n$. Note that every ruling sphere is homologous to $S^2 \times \{*\} \subset S^2 \times S^{d-2}$, for some $S^2 \times S^{d-2}$ connect summand. Thus, for every connect summand $Q_x \cong S^2 \times S^{d-2}$ indexed by $x \in S_n$, let $\omega_x = \mathrm{PD}[ \{*\} \times S^{d-2}] \in H^{2}(Q_x; \mathbf{Z})$. Let 
    $$\omega_n = \sum_{x \in S_n} \omega_x \in \bigoplus_{x \in S_n} H^2(Q_x;\ZZ) \subset H^2(M_n;\ZZ),$$
    where the final inclusion is induced by maps $M_n \to Q_x$ crushing the complement of a connect summand. Then, for any ruling sphere $\Sigma \in \mathfrak{S}_n$, $\langle \omega_n, [\Sigma]\rangle = 1$. Thus, $\omega_n$ tames $\mathfrak{S}_n$, certifying that $\mathfrak{S}_n$ is a ruling. This concludes the proof. 
\end{proof}

\subsection{The Alpert-Balitskiy-Guth example}\label{sec-abg} We review an example due to Balitskiy \cite[Theorem 2.4.4]{balthesis}, elaborated further in \cite[Example 11]{abg24}. This example has subsequently also been discussed by Li-Zhang \cite[Section 2]{lz}. We begin by setting up some notation

Let $\mathbf{R}$ be equipped with the cell structure of a line graph with integer vertices, and $\mathbf{R}^5$ be equipped with the product cell structure. Let ${Z_0} \subset \mathbf{R}^5$ denote the $2$-skeleton with respect to this cell structure. Let $e_1, \cdots, e_5$ denote the coordinate unit vectors, and set $u = e_1 + \cdots + e_5$. Let ${Z_1} := {Z_0} + {u}/2$ denote the $2$-complex dual to ${Z_0}$. Let $\mathbf{Z}^4 \subset \mathbf{Z}^5$ be the subgroup generated by $e_1, \cdots, e_4$.

\begin{definition}Let $\widetilde{M}^4 \subset \mathbf{R}^5$ be the hypersurface given by the locus of points equidistant from ${Z_0}$ and ${Z_1}$, with respect to the $\ell^\infty$-norm on $\mathbf{R}^5$. Equivalently, it is the hypersurface given by the locus of points having $\ell^\infty$-distance $1/4$ from either ${Z_0}$ or ${Z_1}$. Consider the lattices
$$\Lambda_n = (2n\mathbf{Z})^4\oplus \langle u/2 + 2n e_5\rangle \subset \Lambda = (2\mathbf{Z})^4 \oplus \langle u/2 \rangle$$
for any $n > 0$. Let $M = \widetilde{M}/\Lambda$ and $M_n = \widetilde{M}/\Lambda_n$.
\end{definition}

Let $N({Z_0})$ and $N({Z_1})$ denote the $1/4$-neighborhoods in $\ell^\infty$ norm of ${Z_0}$ and ${Z_1}$, respectively. Then, $N({Z_i})$ can be written as a union of the $1/4$-thickening in $\ell^\infty$ norm of the $k$-cells of ${Z_i}$, for $k \in \{0, 1, 2\}$. More precisely, let $I = [-1/4, 1/4]$. For $k \in \{0, 1, 2\}$, any $k$-cell $e^k$ in ${Z_i}$ is a coordinate unit $k$-cube, i.e. $e^k \cong [0,1]^k$. Thus, $e^k$ is a $k$-dimensional $\ell^\infty$ disk of radius $1/2$. Let $e^k_{\circ} \subset e^k$ denote the concentric $\ell^\infty$ sub-disk of radius $1/4$, i.e. $e^k_{\circ}\cong [1/4, 3/4]^k$. Let $B^k = e^k_{\circ} \times I^{5-k}$ denote the normal $1/4$-thickening of $e^k_{\circ}$ in the $(5-k)$ coordinates orthogonal to it. Then, 
$$N({Z_i}) = \bigcup_{0 \leq k \leq 2} \bigcup_{e^k} B^k,$$
where $e^k$ varies over the $k$-cells of ${Z_i}$. Since $B^k$ is a $5$-dimensional PL $k$-handle, $N({Z_i})$ defines a $5$-dimensional PL handlebody with boundary $\partial N({Z_i}) = \widetilde{M}$. Therefore, $\widetilde{M}$ is a PL manifold. Note that the action of $\Lambda$ by translation preserves this PL structure. Therefore, it descends to PL structures on $M$. Since $4$-dimensional PL manifolds have a unique smoothing by \cite[Theorem 2]{milnor}, the PL structure on $M$ can be smoothed in a unique way. We turn the PL embedding $i : M \hookrightarrow \mathbf{R}^5/\Lambda$ into a smooth embedding, by $\varepsilon$-rounding the handles. We equip $M$ with the Riemannian metric $g$ induced from the flat metric on the torus $\mathbf{R}^5/\Lambda$. 

We lift the smooth structure from $M$ to $M_n$ along the covering map $M_n \to M$. Let $g_n$ denote the Riemannian metric on $M_n$ obtained from lifting $g$ along the cover. Furthermore, lifting the smoothed embedding $i : M \hookrightarrow \mathbf{R}^5/\Lambda$, we obtain a smooth embedding $i_n : M_n \hookrightarrow \mathbf{R}^5/\Lambda_n$.

\begin{prop}\label{prop-abg1}
    $M_n$ admits a $\mathrm{Pin}^-$ structure.
\end{prop}

\begin{proof}
    Since $M_n$ is a hypersurface in the $5$-torus $\mathbf{R}^5/\Lambda_n \cong T^5$ and $T^5$ is a spin manifold, it admits a $\mathrm{Pin}^-$ structure by Lemma \ref{lem-pinhyp}.
\end{proof}

\begin{remark}
    $M_n$ is non-orientable. Indeed, the translation by $u/2$ exchanges the interior and exterior of $\widetilde{M}$ in the Euclidean $5$-space, given by the set of points $1/4$-close to $Z_0$ and $Z_1$, respectively. 
\end{remark}

The following is proved by essentially the same argument as in \cite[Example 11]{abg24}.

\begin{prop}\label{prop-HSM}
$\mathrm{HS}(M_n;\ZZ_2) \gtrsim n$.
\end{prop}

\begin{proof}
It suffices to construct a $1$-Lipschitz map of non-zero $\mathbf{Z}_2$-degree from $M_n$ to a round $4$-sphere of radius comparable to $n$. To this end, consider the map $$f : M_n \xrightarrow{i_n} \mathbf{R}^5/\Lambda_n \xrightarrow{j} \mathbf{R}^4/(2n\mathbf{Z})^4 \xrightarrow{k} S^4(n),$$
where $i_n$ is the inclusion, $j$ is given by projection along the vector $v := {u}/2 + 2n {e}_5$ onto the coordinate $4$-torus spanned by ${e}_1, \cdots, {e}_4$, and $k$ is given by crushing the complement of a Euclidean $n$-ball in the flat $4$-torus. 

Evidently, $k$ is $1$-Lipschitz. Note that $i_n$ is a lift of the inclusion $i : M \to \mathbf{R}^5/\Lambda$, and $i$ is $O(1)$-Lipschitz since $M$ is compact. Therefore, the lift $i_n : M_n \to \mathbf{R}^5/\Lambda_n$ is also $O(1)$-Lipschitz. Furthermore, $j$ is $(1+o(1))$-Lipschitz, as $v$ is almost orthogonal to the coordinate plane spanned by ${e}_1, \cdots, {e}_4$, for $n$ sufficiently large. Therefore, $f$ is $O(1)$-Lipschitz. Scaling the target sphere, $f$ defines a $1$-Lipschitz map to a round sphere of radius comparable to $n$. 

Next, we observe that $\deg_{\ZZ_2}(j \circ i_n) \neq 0$. Indeed, fibers of $j$ are circles parallel to $v$ and any such circle intersects $M_n$ an odd number of times. This is because translation by ${u}/2$, and therefore by ${v}$, exchanges the interior and exterior of $\widetilde{M}$ in the Euclidean $5$-space. Since $k$ is evidently degree $1$, $f = k \circ j \circ i_n$ has nonzero $\mathbf{Z}_2$-degree. The conclusion follows. 
\end{proof}

\begin{remark}
    In particular, from Proposition \ref{prop-HSM}, $M$ is $\mathbf{Z}_2$-enlargeable (cf. Definition \ref{def-z2en}).
\end{remark}

We now construct the small rulings on $M_n$. Some of the more technical points in the proof are relegated to Appendix \ref{sec-appendix}.

\begin{prop}\label{prop-abgrule}
    $M_n$ admits a $3/2$-small ruling.
\end{prop}


\begin{proof}
    It suffices to construct a $3/2$-small ruling on $\widetilde{M}$ which is $\Lambda$-invariant (i.e. ruling spheres are sent to ruling spheres upon translation by elements of $\Lambda$), and a taming class represented by a $\Lambda$-invariant cocycle. We now describe the construction of the ruling. Let $B^2 = e^2_\circ \times I^3 \subset N(Z_i)$ be a piecewise linear $2$-handle given by thickening a $2$-cell $e^2 \subset Z_i$, and let $\partial_- B^2 = e^2_\circ \times \partial I^3 \subset \partial N(Z_i)$ denote the co-attaching region of the handle. Observe,
    $$\widetilde{M} = \partial N(Z_i) = \bigcup \partial_- B^2,$$
    where the union runs over all piecewise linear $2$-handles $B^2$ of the handlebody $N(Z_i)$. Note that $\partial_- B^2$ admits a pre-ruling by polyhedral spheres $\{*\} \times \partial I^3$. Therefore, we obtain a pre-ruling $\widetilde{\mathfrak{S}}_i$ on $\widetilde{M}$, for $i = 0, 1$. Note that $\widetilde{\mathfrak{S}}_i$ is $\mathbf{Z}^5$-invariant by construction. The action of $u/2$ exchanges $\widetilde{\mathfrak{S}}_0$ and $\widetilde{\mathfrak{S}}_1$. Therefore, the union $\widetilde{\mathfrak{S}} = \widetilde{\mathfrak{S}}_0 \cup \widetilde{\mathfrak{S}}_1$ is a $\Lambda$-invariant pre-ruling.

    It remains to construct a taming class. Let $r_i : N(Z_i) \to Z_i$ be the $\mathbf{Z}^5$-equivariant deformation retract defined by collapsing each handle to its core, see Lemma \ref{lem-defret}. By a slight abuse of notation, we denote the restriction of $r_i$ to the boundary also as $r_i : M \to Z_i$. The key property of the pre-ruling $\widetilde{\mathfrak{S}}_i$ is that every ruling sphere in $\widetilde{\mathfrak{S}}_i$ maps to a \emph{jail-cell} in $Z_i$, given by the boundary of a $3$-cell in the Euclidean $5$-space, see Lemma \ref{lem-jailrule}. In Proposition \ref{prop-cocycle}, we construct cellular $2$-cocycles $\psi_i \in C^2(Z_i; \mathbf{Z})$ such that for any jail-cell $J \subset Z_i$, $\langle \psi_i, J\rangle = \pm 1$. Let $\tau$ denote the translation on $\mathbf{R}^5$ defined by $\tau(x)= x + u/2$. Then, $r_1 \circ \tau = \tau \circ r_0$ and $\psi_0 \circ \tau= \psi_1$. Consider the cocycle,
    $$\Psi = (r_0)^*\psi_0 + (r_1)^*\psi_1 \in C^2(\widetilde{M}; \mathbf{Z}).$$
    Then, $\Psi \circ \tau = \Psi$. Since $\Psi$ is $\mathbf{Z}^5$ invariant, this shows $\Psi$ is invariant under the subgroup generated by $\mathbf{Z}^5$ and $u/2$. In particular, it is $\Lambda$-invariant. Let $\widetilde{\omega} = [\Psi] \in H^2(\widetilde{M}; \mathbf{Z})$ be the cohomology class represented by $\Psi$. Then, $\widetilde{\omega}$ tames the pre-ruling $\widetilde{\mathfrak{S}}$. Upon quotienting $\widetilde{M}$ by the lattice $\Lambda_n \subset \Lambda$, the $3/2$-small ruling $\widetilde{\mathfrak{S}}$ and the taming class $\widetilde{\omega} \in H^2(\widetilde{M}; \mathbf{Z})$ descend to a $3/2$-small ruling $\widetilde{\mathfrak{S}}_n$ on $M_n$ and a taming class $\omega_n \in H^2(M_n; \mathbf{Z})$. This proves the result.
\end{proof}

\section{Main Theorem}\label{sec-main} 

The following proposition is a synthesis of the results of Section \ref{sec-fclhopf} and Section \ref{sec-strat}. 

\begin{prop}\label{prop-final}
    Let $M^{d-1}$ be a Riemannian manifold satisfying the following conditions: \begin{enumerate}
        \item $M$ admits a $\mathrm{Pin}^-$ structure.
        \item $\mathrm{HS}(M;\ZZ_2) \gtrsim R$.
        \item $M$ admits a $3/2$-small ruling $\mathfrak{S}$.
    \end{enumerate}
Let $\pi: E \to M$ be the circle bundle associated with $\mathfrak{S}$. Then, given any $\sigma > 0$, there is a metric $g_E$ on $E$, such that $(E,g_E)$ has $\msc \geq \sigma$ and $\uw{d-2}(E) \gtrsim R$.
\end{prop}

\begin{proof}
It follows from Theorem \ref{cor-fcl3} that $\uw{d-2}(E) \gtrsim R$. We choose $g_E$ to be an adapted metric (cf. Definition \ref{def-adapted}) on $E$ with radius of the fibers $\rho = \rho(\sigma) > 0$ sufficiently small. Then, it follows from Theorem \ref{thm-circbdlmscal} that $(E, g_E)$ satisfies $\msc \geq \sigma$.
\end{proof}

\begin{remark}
    In Section \ref{sec-construction}, we have given examples of manifolds satisfying the hypotheses of Proposition \ref{prop-final}. One can generate higher dimensional examples of manifolds satisfying these hypotheses from existing examples by taking products with circles of radius $R \gg 0$.
\end{remark}

We can now prove the main result of the paper. 

\begin{theorem}\label{thm-final} 
Conjecture \ref{conj-abgtrue} is false in dimensions four and above. 
\end{theorem}
\begin{proof}
Consider the manifolds $M_n$ constructed in Section \ref{sec-corrugration} or Section \ref{sec-abg}. By the results established therein, $M_n$ satisfy the hypotheses of Proposition \ref{prop-final} with $R=n$. Thus, applying Proposition \ref{prop-final}, we obtain that for any given $\sigma > 0$, there exists Riemannian manifolds $E_n$ with $\msc \geq \sigma$ and codimension two Urysohn width $\gtrsim n$. Thus, for $n$ sufficiently large, $E_n$ furnish counterexamples to Conjecture \ref{conj-abgtrue}. 
\end{proof}

\appendix
\section{Some details concerning Section \ref{sec-abg}}\label{sec-appendix}

In this section, we explicate certain points in the proof of Proposition \ref{prop-abgrule} for completeness. Let us begin by recalling the $\mathbf{Z}^5$-equivariant retracts $r_i : N(Z_i) \to Z_i$ defined therein.

\begin{lemma} \label{lem-defret}
    There exists a $\mathbf{Z}^5$-equivariant deformation retract $r_i  : N({Z_i}) \to {Z_i}$, where $N({Z_i})$ denotes the $1/4$-neighborhood of ${Z_i}$ in $\ell^\infty$-norm. 
\end{lemma}

\begin{proof}
    Recall $N({Z_i}) = \bigcup_k \bigcup_{e^k} B^k$, where $B^k = e^k_{\circ} \times I^{5-k}$. Let $r_i|_{B^k}$ be defined by the composition $B^k \to e^k_{\circ} \to e^k$, where the first map is the coordinate projection and the second map is scaling by $2$. This gives the desired deformation retract $r_i : N({Z_i}) \to {Z_i}$.
\end{proof}

\begin{lemma}\label{lem-retrhom} $(r_0|_M)_* \oplus (r_1|_M)_* : H_2(\widetilde{M}) \to H_2({Z_0}) \oplus H_2({Z_1})$ is an isomorphism.
\end{lemma}

\begin{proof}
   Let $\iota_i : \widetilde{M} = \partial N(Z_i) \hookrightarrow N(Z_i)$, $i = 0, 1$ denote the inclusions as the boundary. Note that we have $N(Z_0) \cap N(Z_1) = \RR^5$ and $N(Z_0) \cap N(Z_1) =\widetilde{M}$. Therefore, the Mayer-Vietoris sequence gives an isomorphism 
   $$(\iota_0)_*\oplus (\iota_1)_* : H_2(\widetilde{M}) \to H_2(N(Z_0)) \oplus H_2(N(Z_1))$$
   Since $r_i$ are deformation retracts, $(r_i)_* : H_2(N(Z_i)) \to H_2(Z_i)$ are isomorphisms. Since $r_i|_M = r_i \circ \iota_i$, this proves $(r_0|_M)_* \oplus (r_1|_M)_*$ is an isomorphism, as well.
\end{proof}

We introduce the following convenient terminology

\begin{definition}\label{def-jailcell}
    A \emph{jailcell} in ${Z_0}$ (resp. ${Z_1}$) is the boundary of a unit $3$-cube in the Euclidean $5$-space, with vertices given by vectors with integral (resp. half-integral) coordinates. The fundamental class of the jailcells generate $H_2({Z_i})$. We say a class in $H_2(\widetilde{M})$ is a \emph{jailcell class} if its image under the isomorphism
    $$(r_0|_M)_* \oplus (r_1|_M)_* : H_2(\widetilde{M}) \to H_2({Z_0}) \oplus H_2({Z_1}),$$
    is the fundamental class of a unique jailcell belonging to either ${Z_0}$ or ${Z_1}$.
\end{definition}

Recall the pre-ruling $\widetilde{\mathfrak{S}} = \widetilde{\mathfrak{S}}_0 \cup \widetilde{\mathfrak{S}}_1$ of $\widetilde{M}$ as defined in the proof of Proposition \ref{prop-abgrule}.

\begin{lemma}\label{lem-jailrule}
    The ruling spheres of $\widetilde{\mathfrak{S}}$ represent jailcell classes in $H_2(\widetilde{M})$.
\end{lemma}

\begin{proof}
    Let $e^2 \subset Z_0$ be a $2$-cell, and $\{*\} \times \partial I^3 \subset e^2_\circ \times \partial I^3 = \partial_-B^2$ be a ruling sphere from $\widetilde{\mathfrak{S}}_0$. Let us ensure by a homotopy that $* \in e^2_{\circ}$ is the center of the $2$-cell. Let $J^\vee \subset Z_1$ be the jailcell given by the boundary of the unit $3$-cube orthogonal to $e^2$ with center $*$. Then, $r_1(\{*\} \times \partial I^3) = J^\vee$. This shows ruling spheres of $\widetilde{\mathfrak{S}}_0$ represent jailcell classes. The argument for $\widetilde{\mathfrak{S}_1}$ goes through mutatis mutandis using the retract $r_0$. This concludes the proof.
\end{proof}

We now come to the main result of this section. Let us consider the index two subgroup 
$\Gamma = (2\mathbf{Z})^4 \oplus \langle u\rangle \subset \Lambda$.
Let $\tau$ be the translation on $\mathbf{R}^5$ defined by $\tau(x) = x + u/2$. Note that $\tau(Z_0) = Z_1$, and $\tau(\widetilde{M}) = \widetilde{M}$. Furthermore, $\tau|_{Z_0} : Z_0 \to Z_1$ and $\tau|_{\widetilde{M}}$ are cellular.

\begin{prop}\label{prop-cocycle}
    There exist a $\Gamma$-invariant cellular $2$-cocycle $\psi_i \in C^2(Z_i; \mathbf{Z})$ such that for any jailcell $J \subset Z_i$, $\langle \psi_0, J\rangle = \pm 1$. Further, $\psi_1 = \psi_0 \circ \tau$.
\end{prop}

\begin{proof}
    Note that every $2$-cell in ${Z_0}$ is determined by an integral lattice point $v \in \mathbf{Z}^5$ and two indices $1 \leq i < j \leq 5$. The $2$-cell corresponding to $(v, i, j)$ is spanned by $v + {e}_i$ and $v + {e}_j$, where ${e}_1, \cdots, {e}_5$ denotes the standard unit vectors in $\mathbf{R}^5$. We define,
    $$\psi_0(v, i, j) = 
    \begin{cases} 
    0, & \mathrm{if}\; j = 5,\\
    v_2 + v_3 + v_4 + v_5 \pmod{2}, & \mathrm{if}\; i = 1, j = 2,\\
    v_4 + v_5 \pmod{2} & \mathrm{otherwise}
    \end{cases}$$
    We extend $\psi_0$ $\mathbf{Z}$-linearly to the cellular chain group $C_2({Z_0})$. This defines a $2$-cochain on ${Z_0}$. Since $Z_0$ has no $3$-cells, $\psi_0$ is automatically a cellular $2$-cocycle. Observe that, by parity, $\psi_0$ is $\Gamma$-invariant. Next, let $J$ be a jailcell, given by the oriented boundary of the unit $3$-cube spanned by the vectors $v + {e}_i, v + {e}_j, v + {e}_k$ for $1 \leq i < j < k \leq 5$. Then, 
    $$\langle \psi_0, J \rangle = \sum_{cyclic} \psi_0(v+{e}_k,i,j) - \psi_0(v,i,j),$$
    where the sum is over all cyclic permutations of $(i, j,k)$. We verify that this sum evaluates to $1$ on a case by case basis. First, suppose $k = 5$. Then, we have
    \begin{align*}\langle \psi_0, J\rangle &= \psi_0(v + {e}_5, i, j) - \psi_0(v, i,j)
    \\& =
    \begin{cases}
        (v_2 + v_3 + v_4 + v_5 + 1) - (v_2 + v_3 + v_4 + v_5) = 1, & \mathrm{if}\; i = 1, j = 2 \\
        (v_4 + v_5 + 1) - (v_4 + v_5) = 1, & \mathrm{otherwise} 
    \end{cases}
    \end{align*}
    {Now,} suppose $\{i, j, k\} \subset \{1, 2, 3, 4\}$. {Consider the case} $i = 1$, $j = 2$ and $k \in \{3, 4\}$. Then, 
    \begin{align*}
        \psi_0({v} + e_k, 1, 2) - \psi_0({v}, 1, 2) &= (v_2 + v_3 + v_4 + v_5 + 1) - (v_2 + v_3 + v_4 + v_5) = 1\\
        \psi_0({v} + e_2, 1, k) - \psi_0({v}, 1, k) &= (v_4 + v_5) - (v_4 + v_5) = 0 \\
        \psi_0(v + e_1, 2, k) - \psi_0(v, 2, k) &= (v_4 + v_5) - (v_4 + v_5) = 0
    \end{align*}
    Hence, $\langle \psi_0, J\rangle = 1$ in this case. {Next, consider the case} $i \in \{1, 2\}$, $j = 3$, $k = 4$. Then,
    \begin{align*}
        \psi_0(v + e_4, i, 3) - \psi_0(v, i, 3) &= (v_4 + 1 + v_5) - (v_4 + v_5) = 1 \\
        \psi_0(v + e_3, i, 4) - \psi_0(v, i, 4) &= (v_4 + v_5) - (v_4 + v_5) = 0 \\
        \psi_0(v + e_i, 3, 4) - \psi_0(v, 3, 4) &= (v_4 + v_5) - (v_4 + v_5) = 0
    \end{align*}
    Hence, $\langle \psi_0, J\rangle = 1$ in this case as well. This finishes the proof.
\end{proof}

\bibliography{main}
\bibliographystyle{alpha}

\end{document}